\documentclass{article}

\usepackage{amsmath}
\usepackage{amssymb}
\usepackage{amsthm,mathtools}
\usepackage{bbm}
\usepackage{hyperref}
\usepackage{theoremref}
\usepackage{enumitem}

\newcommand{\ap}[1]{\prescript{a}{}{#1}}

\def\reals{\mathbbm{R}}
\def\ereals{\overline{\reals}}

\def\comp{\mathop{\text{\scriptsize\raise 1pt \hbox{$\circ$}}}}
\def\infconv{\mathop{\text{\scriptsize\raise 1pt \hbox{$\square$}}}}

\def\argmin{\mathop{\rm argmin}\limits}

\def\minimize{\mathop{\rm minimize}\limits}
\def\maximize{\mathop{\rm maximize}\limits}
\def\over{\quad\mathop{\rm over}\quad}

\def\st{\mathop{\rm subject\ to}}

\def\dom{\mathop{\rm dom}\nolimits}

\def\ovr{\mathop{\rm over}\ }

\def\upto{{\raise 1pt \hbox{$\scriptstyle \,\nearrow\,$}}}
\def\downto{{\raise 1pt \hbox{$\scriptstyle \,\searrow\,$}}}

\def\cl{\mathop{\rm cl}\nolimits}

\def\lin{\mathop{\rm lin}}

\def\epi{\mathop{\rm epi}\nolimits}

\def\tos{\rightrightarrows}
\def\FF{(\F_t)_{t=0}^T}

\def\A{{\cal A}}

\def\C{{\cal C}}

\def\F{{\cal F}}
\def\G{{\cal G}}

\def\L{{\cal L}}
\def\M{{\cal M}}
\def\N{{\cal N}}

\def\S{{\cal S}}

\def\T{{\cal T}}
\def\U{{\cal U}}
\def\V{{\cal V}}
\def\X{{\cal X}}
\def\Y{{\cal Y}}

\newtheorem{theorem}{Theorem}
\newtheorem{lemma}[theorem]{Lemma}
\newtheorem{corollary}[theorem]{Corollary}

\newtheorem{example}[theorem]{Example}
\newtheorem{cexample}[theorem]{Counterexample}
\newtheorem{remark}[theorem]{Remark}

\theoremstyle{definition}

\newtheorem{assumption}[theorem]{Assumption}

\title{Duality in convex stochastic optimization}

\author{Teemu Pennanen\thanks{Department of Mathematics, King's College London, Strand, London, WC2R 2LS, United Kingdom, teemu.pennanen@kcl.ac.uk} \and Ari-Pekka Perkki\"o\thanks{Mathematics Institute, Ludwig-Maximilian University of Munich, Theresienstr. 39, 80333 Munich, Germany, a.perkkioe@lmu.de. Corresponding author}}

\begin{document}

\maketitle

\begin{abstract}
This paper studies duality and optimality conditions in general convex stochastic optimization problems introduced by Rockafellar and Wets in \cite{rw76}. We derive an explicit dual problem in terms of two dual variables, one of which is the shadow price of information while the other one gives the marginal cost of a perturbation much like in classical Lagrangian duality. Existence of primal solutions and the absence of duality gap are obtained without compactness or boundedness assumptions. In the context of financial mathematics, the relaxed assumptions are satisfied under the well-known no-arbitrage condition and the reasonable asymptotic elasticity condition of the utility function. We extend classical portfolio optimization duality theory to problems of optimal semi-static hedging. Besides financial mathematics, we obtain several new frameworks in stochastic programming and stochastic optimal control.

\end{abstract}


\section{Introduction}

Given a probability space $(\Omega,\F,P)$ with a filtration $\FF$ (an increasing sequence of sub-$\sigma$-algebras of $\F$), consider the problem
\begin{equation}\label{sp}\tag{$SP$}
  \begin{aligned}
    &\minimize\quad & & Ef(x,\bar u):=\int f(x(\omega),\bar u(\omega),\omega)dP(\omega)\quad\ovr\text{$x\in\N$}%\\
%    &\st & & x\in\N,
  \end{aligned}
\end{equation}
where $\N$ is a linear space of stochastic processes $x=(x_t)_{t=0}^T$ adapted to $\FF$ (i.e., $x_t$ is $\F_t$-measurable) and $\bar u$ is a $\reals^m$-valued random variable. % and $Ef$ is an integral functional on the space $L^0(\Omega,\F,P;\reals^n\times\reals^m)$ of $\reals^n\times\reals^m$-valued random variables. 
%\[
%Ef(x,u) := \int_\Omega f(x(\omega),u(\omega),\omega)dP(\omega).
%\]
We assume that $f$ is a {\em convex normal integrand} on $\reals^n\times\reals^m\times\Omega$, i.e.\ $f(\cdot,\omega)$ is a closed convex function for every $\omega\in\Omega$ and $\omega\mapsto\epi f(\cdot,\omega)$ is an $\F$-measurable set-valued mapping; see \cite[Chapter~14]{rw98}. Here and in what follows, we define the integral of an extended real-valued random variable as $+\infty$ unless its positive part is integrable. The integral of any extended real-valued measurable function is then a well defined extended real number so it follows that $Ef$ is a well-defined convex function on $L^0(\reals^n\times\reals^m)$.

Problems of the form \eqref{sp} were first studied in \cite{rw76} where it was observed that many more specific stochastic optimization problems can be written in this unified format. Examples include more traditional formulations of stochastic programming, convex stochastic control and various problems in financial mathematics; see Section~\ref{sec:appdual} below. In \cite{rw76}, problem \eqref{sp} was analyzed through dynamic programming and convex duality. Soon after, \cite{evs76} extended the dynamic programming principle by removing the convexity assumption but, like \cite{rw76}, assumed the set of feasible solutions to be bounded. The boundedness assumptions were removed in \cite{pen11c,pp12,per16,bpp18,ppr16,pp22}. 

Like \cite{rw76,pli82,pen11c,bpp18}, the present paper studies problem \eqref{sp} with the functional analytic techniques of convex duality. This will yield dual problems whose optimum values coincide with that of \eqref{sp} and whose optimal solutions can be used to characterize those of \eqref{sp}. We extend the classic results of \cite{rw76,rw78,rw83} so as to cover various duality results developed independently in stochastics and financial mathematics e.g.\ in \cite{dk94,sch92,kab99,sch4,pp10}. The new results allow also for significant extensions to central models in stochastic programming, stochastic optimal control and financial mathematics. In particular, we extend the inequality constrained models of \cite{rw78} by including equality constraints and allowing for unbounded strategies. In stochastic optimal control, we obtain a scenariowise maximum principle. We also extend the classical duality results of financial mathematics to optimal semistatic hedging problems where one optimizes over dynamic trading strategies as well as statically held derivative portfolios. In each application, we establish the existence of primal solutions and the absence of a duality gap.

%Convex duality is based on embedding a given optimization problem in a family of optimization problems parameterized by an element of a locally convex topological vector space; see Section~\ref{sec:cd}.

Much like in \cite{rw76,rw78,pen11c,bpp18}, our strategy is to analyze \eqref{sp} through the general duality framework of \cite{roc74}. We deviate from the above references, however, in that we employ two dualizing parameters: the random vector $\bar u$ in \eqref{sp} and another one that perturbs the adaptedness constraint on $x$. This yields an explicit dual problem for \eqref{sp} in terms of two dual variables: one is the ``shadow price of information'' studied e.g.\ in \cite{wet75,rw76,pli82,dav92,db92,pp18a} and the other one gives the marginal cost of changing $\bar u$. As a special case, we obtain the dual problem of \cite{rw78} for stochastic optimization problems with inequality constraints. We find new duality frameworks for many other problem classes including optimal stoping, stochastic optimal control and portfolio optimization. Moreover, our results apply without the compactness and boundedness assumptions made in \cite{rw78}.

Without the boundedness assumptions, problem \eqref{sp} does not directly fit the framework of \cite{roc74} which assumes that the optimal solutions are sought from a locally convex vector space. We will thus first, in Section~\ref{sec:dual}, restrict the decision strategies $x$ to a locally convex space $\X$ of $\reals^n$-valued random variables. Straightforward application of the functional analytic duality theory then yields a dual problem and optimality conditions for the restricted problem. %The optimality conditions turn out to be scenariowise conditions that reduce to the classical Karush-Kuhn-Tucker conditions in the deterministic setting. %After deriving the duality framework for strategies in a locally convex space in Section~\ref{sec:dual},
We return to the original problem \eqref{sp} in Section~\ref{sec:rel} and find that its optimum value as a function of the parameters $(z,u)$ has the same lower semicontinuous hull as that of the restricted problem. It follows that their dual problems coincide and, by an application of Fenchel inequality and \thref{lem:perp}, we find scenariowise optimality conditions for \eqref{sp}. Section~\ref{sec:adg} recalls sufficient conditions for the lower semicontinuity of the optimum value function of \eqref{sp}. Section~\ref{sec:appdual} illustrates the new results with applications to more specific problems classes.

%In many duality results in the literature, the dual problems and optimality conditions are given in terms of the dual variable $y$ only. Under certain measurability conditions, these can be recovered from our general duality framework by considering a reduced dual problem obtained by optimizing the dual objective over the shadow price of information for a given $y$.

\section{Integral functionals in duality}\label{sec:ifd}

Convex duality is based on the theory of conjugate functions on dual pairs of locally convex topological vector spaces; see \cite{roc74}. The first part of this section reviews spaces of random variables in separating duality with each other while the second part reviews conjugation of integral functionals on such spaces. This forms the functional analytic setting for the duality theory of stochastic optimization developed in the followup sections. For full generality, we make minimal assumptions on the spaces of random variables. The classical Lebesgue and Orlicz spaces, $L^p$ and $L^\Phi$ are covered as special cases but also many others that come up naturally e.g.\ in engineering and finance.

\subsection{Dual spaces of random variables}\label{sec:dsrv}

Let $\U$ and $\Y$ be linear spaces of $\reals^m$-valued random variables in separating duality under the bilinear form
\[
\langle u,y\rangle := E[u\cdot y].
\]
This means that $u\cdot y\in L^1$ for all $u\in\U$ and $y\in\Y$ and that for every nonzero $u\in\U$, there exists a $y\in\Y$ such that $\langle u,y\rangle\ne 0$ and vice versa. As usual, we identify random variables that coincide almost surely so the elements of $\U$ and $\Y$ are actually equivalence classes of random variables that coincide almost surely. We will also assume that the spaces are {\em decomposable} and {\em solid}. Decomposability means that
\[
1_Au+1_{\Omega\setminus A}u'\in\U
\]
for every $u\in\U$ and $u'\in L^\infty$ while solidity means that if $\bar u\in\U$ and $u\in L^0$ are such that $|u^i|\le|\bar u^i|$ almost surely for every $i=1,\ldots,m$, then $u\in\U$; similarly for $\Y$. Solidity implies that
\[
\U=\U_1\times\cdots\times\U_m\quad\text{and}\quad \Y=\Y_1\times\cdots\times\Y_m,
\]
where $\U_i$ and $\Y_i$ are solid decomposable spaces of real-valued random variables in separating duality under the bilinear form $(u_i,y_i)\mapsto E[u_iy_i]$. In particular,
\begin{equation}\label{eq:solidci}
u_iy_i\in L^1\quad\text{and}\quad \langle u,y\rangle = \sum_{i=1}^mE[u_iy_i]\quad\forall u\in\U,y\in\Y.
\end{equation}
Given a solid space of real-valued random variables $\U_0$, the space  $\{u\in L^0(\reals^m)\mid |u|\in\U_0\}$ is solid and it can be written as $\U_0^m$, the $m$-fold Cartesian product of $\U_0$. A solid space containing all constant functions is decomposable. The following shows that the converse does not hold.
\begin{cexample}\label{ex:decomp}
Let $x\ge 1$ be an unbounded real-valued random variable and $\X:=L^\infty +Lin (x1_A\mid A\in\F)$. Then $\X$ is decomposable, by construction,  but not solid, since it does not contain $\sqrt{x}$.
\end{cexample}

Decomposable solid spaces of random variables in separating duality include Lebesgue spaces, Orlicz spaces, Marcinkiewich spaces paired with Lorentz spaces, spaces of finite moments $\|u\|_{L^p}$ for all $p\in(1,\infty)$ as well as the general class of Banach Function Spaces or, even more generally, locally convex function spaces; see \cite{pp220} and its references. The spaces of continuous functions or various Sobolev spaces of functions on $\reals^n$ fail to be decomposable or solid. The space $L^0$ of all random variables is decomposable and solid but if $(\Omega,\F,P)$ is atomless, it  cannot be paired with a nontrivial space of random variables. Indeed, if $y\in L^0$ is nonzero, then there exists $\epsilon>0$ and $A\in\F$ such that $|y|1_A> \epsilon$ and $P(A)$. Since the space is atomless, there exists $\eta<0$ with $E[1_A \eta]=-\infty$.  Choosing $u=1_A y \eta$, we get $E[u\cdot y]=-\infty$.

Given a topology on $\U$, the corresponding topological dual of $\U$ is the linear space of all continuous linear functionals on $\U$. A topology is {\em compatible} with the bilinear form on $\U\times\Y$ if every continuous linear functional can be expressed in the form
\[
u\mapsto\langle u,y\rangle
\]
for some $y\in\Y$. Such topologies can be characterized in terms of the ``weak'' and ``Mackey'' topologies associated with the bilinear form. The {\em weak topology} $\sigma(\U,\Y)$ on $\U$ is the topology generated by linear functionals $u\mapsto\langle u,y\rangle$ where $y\in\Y$. Similarly for $\Y$. The {\em Mackey topology} is the topology generated by the sublinear functionals
\[
\sigma_D(u):=\sup_{y\in D}\langle u,y\rangle,
\]
where $D\subset\Y$ is $\sigma(\Y,\U)$-compact. Similarly for $\Y$. Given a topology on $\U$, the corresponding topological dual can be identified with $\Y$ if and only if the topology is between $\sigma(\U,\Y)$ and $\tau(\U,\Y)$. If $\U$ is Fr\'echet (e.g.\ Banach) and $\Y$ is its topological dual, then the $\sigma(\Y,\U)$-compact sets are the bounded sets in $\Y$, so $\tau(\U,\Y)$ is the strong topology; see \cite{kn76}.

The following is from \cite{pp12}.
\begin{lemma}\thlabel{lem:rel}
We have $L^\infty\subseteq \U \subseteq L^1$ and $L^\infty\subseteq \Y\subseteq L^1$ and 
\begin{align*}
\sigma(L^1,L^\infty)|_\U&\subseteq\sigma(\U,\Y),\quad\sigma(\U,\Y)|_{L^\infty}\subseteq\sigma(L^\infty,L^1),\\
\tau(L^1,L^\infty)|_\U&\subseteq\tau(\U,\Y),\quad\tau(\U,\Y)|_{L^\infty}\subseteq\tau(L^\infty,L^1).
\end{align*}
The $L^0$-topology on $\U$ is weaker than $\tau(\U,\Y)$.
\end{lemma}

Given a decomposable space $\U$, its {\em K\"othe dual} is the linear space
\[
\{y\in L^0\mid u\cdot y\in L^1\quad \forall u\in\U\}.
\]
This is the largest space of random variables that can be paired with $\U$ with the bilinear form $(u,y)\mapsto E[u\cdot y]$. Clearly, the K\"othe dual of a solid space is solid. The following is well-known, e.g., in Lebesgue and Orlicz spaces.

\begin{lemma}\thlabel{lem:EGcont3}
Let $\U$ and $\Y$ be solid and $\G\subset\F$ a $\sigma$-algebra such that $E^\G\U\subset\U$. The mapping $E^\G:\U\to\U$ is weakly continuous if and only if $E^\G\Y\subset\Y$ and in this case,
\[
\langle E^\G u,y\rangle = \langle u,E^\G y\rangle\quad\forall u\in\U,\ y\in\Y.
\]
If $\Y$ is the K\"othe dual of $\U$, then $E^\G\Y\subset\Y$.
\end{lemma}

\begin{proof}
If $u^i$, $y^i$, $(E^\G u)^iy^i$ and $u^i(E^\G y)^i$ are integrable, \thref{lem:ce} gives
\begin{equation}\label{l1adjoint}
E[E^\G u\cdot y] = E[(E^\G u)\cdot E^\G y] = E[u\cdot E^\G y].
\end{equation}
Thus, if $E^\G\U\subset\U$ and $E^\G\Y\subset\Y$, then, by \eqref{eq:solidci}, the function $u\mapsto E^\G u$ is weakly continuous. On the other hand, if $E^\G:\U\to\U$ is weakly continuous, then  $u\mapsto E[ E^\G u\cdot y]$ is $\sigma(\U,\Y)$-continuous for $y\in\Y$. Thus, there exists a $y'\in\Y$ such that $E[E^\G u\cdot y]=E[u\cdot y']$ for all $u\in\U$. Since $y\in L^1$, \eqref{l1adjoint} gives
\[
E[E^\G u\cdot y]=E[u\cdot E^\G y]\quad\forall u\in L^\infty.
\]
Thus, $y'=E^\G y$ almost surely.

Assume now that $\Y$ is the K\"othe dual of $\U$ and let $y\in\Y$. It suffices to show $E^\G y\in \Y$. By solidity and linearity, we may assume that at most one component $y^i$ of $y$ is nonzero and that it is nonnegative. Then $E^\G y^i$ is nonnegative. Since $\Y$ is the K\"othe dual, it suffices to show that  $E[u^i (E^\G y^i)]<\infty$ for every nonnegative $u\in\U$. By \thref{lem:ce},  $E[u^i (E^\G y^i)]= E[E^\G (u^i) y^i]$, where the right side is finite, since $E^\G\U\subset\U$.
\end{proof}

Let $\X$ and $\V$ be decomposable solid spaces of $\reals^n$-valued random variables in separating duality under the bilinear form
\[
(x,v)\mapsto E[x\cdot v].
\]
A linear mapping $\A:\X\to\U$ is {\em weakly continuous} if it is continuous with respect to the weak topologies. This means that $x\mapsto\langle \A x,y\rangle$ is $\sigma(\X,\V)$-continuous for all $y\in\Y$, or equivalently, there exists a linear mapping $\A^*:\Y\to\V$ such that
\[
\langle \A x,y\rangle = \langle x,\A^* y\rangle \quad\forall x\in\X,\ y\in\Y.
\]
The mapping $\A^*$ is known as the {\em adjoint} of $\A$.

\begin{lemma}\thlabel{lem:A0}
Let $A\in L^0(\reals^{m\times n})$ be a random matrix such that $A x\in\U$ for all $x\in\X$. The linear mapping $\A:\X\to\U$ defined pointwise by
\[
\A x=Ax\quad a.s.
\]
is weakly continuous if and only if  $A^*y\in\V$ for all $y\in\Y$, and in this case its adjoint is given pointwise by
\[
\A^*y=A^*y\quad a.s.
\]
If $\V$ is the K\"othe dual of $\X$, then $A^*y\in\V$ for all $y\in\Y$,
\end{lemma}

\begin{proof}
For any $x\in\X$ and $y\in\Y$,
\[
\langle \A x,y\rangle = E[(Ax)\cdot y] = E[x\cdot A^*y],
\]
which proves the equivalence and the adjoint formula. The above equation implies that $x\cdot A^*y\in L^1$, so $A^*y\in\V$ when $\V$ is the K\"othe dual of $\X$.
\end{proof}

\subsection{Conjugates of integral functionals}

This section studies convex integral functionals on paired decomposable spaces $\U$ and $\Y$ of random variables. More precisely, we take a normal integrand $h$ and study the integral functionals $Eh:\U\to\ereals$ and $Eh^*:\Y\to\ereals$ defined by
\[
Eh(u) := \int_\Omega h(u(\omega),\omega)dP(\omega)
\]
and
\[
Eh^*(y) := \int_\Omega h^*(y(\omega),\omega)dP(\omega).
\]

The following two theorems are essentially reformulations of the main results in \cite{roc68}. We give the simple proofs for completeness.

\begin{theorem}\thlabel{thm:cif0}
If $h$ is a convex normal integrand with $\dom Eh\ne\emptyset$, then  
\[
(Eh)^*=Eh^*.
\]
Moreover, $y\in\partial Eh(u)$ if and only if $Eh(u)$ is finite and $y\in\partial h(u)$ almost surely.
\end{theorem}

\begin{proof}
The first claim follows by applying \cite[Theorem~14.60]{rw98} to the normal integrand $h_y(u,\omega):=h(u,\omega)-u\cdot y(\omega)$, where $y\in\Y$. As to the second, we have $y\in\partial Eh(u)$ if and only if $Eh(u)$ is finite and $Eh(u)+(Eh)^*(y)=\langle u,y\rangle$. Since $(Eh)^*=Eh^*$ by the first part,  the equality holds, by Fenchel's inequality, if and only if
\[
h(u)+h^*(y)=u\cdot y
\]
almost surely. This means that $y\in\partial h(u)$ almost surely.
\end{proof}

\begin{corollary}\thlabel{cor:cif}
Let $h$ be a convex normal integrand. The following are equivalent
\begin{enumerate}
\item $\dom Eh \ne\emptyset$ and  $\dom E h^* \ne\emptyset$,
\item $Eh$ is proper and closed,
%\item $Eh^*$ is proper and lsc,
\item $\dom Eh\ne\emptyset$ and there exists $y\in\Y$ and $\alpha\in L^1$ such that
\begin{align*}
h(u,\omega)\ge u\cdot y(\omega)-\alpha(\omega)
\end{align*}
\end{enumerate}
and imply that  $Eh$ and $E h^*$ are conjugates of each other and that $y\in\partial Eh(u)$ if and only if $y\in\partial h(u)$ almost surely.
\end{corollary}

\begin{proof}
By \thref{thm:cif0}, 1 implies 2. Assuming 2, there exists $u\in\dom Eh$, $y\in\Y$ and $a\in\reals$ such that
\[
Eh(u)\ge \langle u,y\rangle-a\quad\forall u\in\U
\]
Thus, by \thref{thm:cif0},
\[
a \ge (Eh)^*(y)=Eh^*(y).
\]
By Fenchel's inequality
\[
h(u,\omega)+h^*(y,\omega)\ge u\cdot y,
\]
so 3 holds with $\alpha(\omega)=h^*(y(\omega),\omega)$. If 3 holds, $Eh^*(y)\le E\alpha$, so 1 holds. By \thref{thm:cif0}, $Eh$ and $Eh^*$ are conjugates of each other and $y\in\partial Eh(u)$ implies $y\in\partial h(u)$ almost surely. If $y\in\partial h(u)$, then $h(u)+h^*(y)=u\cdot y$ almost surely, where each summand is integrable, since $Eh$ and $Eh^*$ are proper, so $Eh(u)+Eh^*(y)=\langle u,y\rangle$, which means that $y\in\partial Eh(u)$.
\end{proof}

\begin{corollary}\thlabel{cor:scl}
Given a closed convex valued measurable mapping $S:\Omega\tos\reals^m$, the set
\[
\S := \{u\in\U\mid u\in S\ \text{a.s.}\}
\]
is closed and convex.
\end{corollary}
\begin{proof}
This follows by applying \thref{thm:cif0} to the conjugate of $h(u,\omega):=\delta_S(u,\omega)$.
\end{proof}

By symmetry, one can add obvious dual versions of 2 and 3 in the list of equivalent conditions in \thref{cor:cif}. The following gives a general form of the classical Jensen's inequality for conditional expectations.

\begin{theorem}[Jensen's inequality]\thlabel{thm:jensen}
Assume that $\U$ and $\Y$ are solid with $E^\G\U\subset\U$ and $E^\G\Y\subset\Y$ and let $h$ be a $\G$-measurable convex normal integrand such that $Eh^*$ is proper on $\Y$. Then
\[
Eh(E^\G u)\le Eh(u)
\]
for every $u\in\U$.
\end{theorem}

\begin{proof}
Assume first that $Eh^*$ is proper on $\Y\cap L^0(\G)$. \thref{cor:cif} then gives
\begin{align*}
Eh(E^\G u) &=\sup_{y\in\Y\cap L^0(\G)}\{E[(E^\G u)\cdot y]-Eh^*(y)\}\\
&= \sup_{y\in\Y\cap L^0(\G)}E[u\cdot y-h^*(y)]\\
&\le E\sup_{y\in\reals^m}\{u\cdot y - h^*(y)\}\\
&=Eh(u)
\end{align*}
for any $u\in\U$. If $Eh^*$ is proper merely on $\Y$, then $E(h^*)^+$ is proper on $\Y$ as well. The function $E[(h^*)^+]^*$ is finite at the origin so, by the first part of the proof, $E(h^*)^+(E^\G y)\le E(h^*)^+(y)$, so $Eh^*$ is finite on $\Y\cap L^0(\G)$. 
\end{proof}

\section{Duality for integrable strategies}\label{sec:dual}

We will develop a duality theory for \eqref{sp} by applying the general conjugate duality framework of Rockafellar \cite{roc74} first to the parametric optimization problem
%\begin{equation}\label{pinfty}\tag{$SP^\X_u$}
%\minimize\quad Ef(x,u):=\int f(x(\omega),u(\omega),\omega)dP(\omega)\quad\ovr\text{$x\in\N\cap\X$},
%\end{equation}
\begin{equation}\label{spx}\tag{$SP_\X$}
  \begin{aligned}
    &\minimize\quad & & Ef(x,\bar u):=\int f(x(\omega),\bar u(\omega),\omega)dP(\omega)\quad\ovr\text{$x\in\X_a$}%\\
%    &\st & & x\in\N,
  \end{aligned}
\end{equation}
where $\X\subset L^0(\Omega,\F,P;\reals^n)$ is a solid decomposable space of random paths and
\[
\X_a:=\X\cap\N.
\]
We will assume that the parameter $\bar u$ belongs to another solid decomposable space $\U\subset L^0(\Omega,\F,P;\reals^m)$ of random variables. The general theory of convex duality will give a dual problem and optimality conditions for \eqref{spx}. Section~\ref{sec:rel} will then extend these results to the original problem \eqref{sp} where we optimize over general adapted strategies in $L^0$.

We embed \eqref{spx} into the conjugate duality framework by introducing an additional parameter $z\in\X$ and the extended real-valued convex function $F$ on $\X\times\X\times\U$ defined by
\[
F(x,z,u) := Ef(x,u) + \delta_\N(x-z).
\]
We denote the associated optimum value function by
\[
\varphi(z,u):=\inf_{x\in\X}\{Ef(x,u)\,|\,x-z\in\N\}.
\]
We assume that $\X$ is in separating duality with a solid decomposable space $\V\subset L^0(\Omega,\F,P;\reals^n)$ and that $\U$ is in separating duality with a solid decomposable space $\Y\subset L^0(\Omega,\F,P;\reals^m)$. The bilinear forms are the usual ones, i.e.
\[
\langle x,v\rangle:=E[x\cdot v]\quad\text{and}\quad \langle u,y\rangle := E[u\cdot y].
\]
Solidity implies that
\[
\X = \X_0\times\cdots\times\X_T\quad\text{and}\quad \V = \V_0\times\cdots\times\V_T,
\]
where $\X_t$ and $\V_t$ are solid decomposable spaces of $\reals^{n_t}$-valued random variables in separating duality under the bilinear form $(x_t,v_t)\mapsto E[x_t\cdot v_t]$. It follows that
\[
\langle x,v\rangle = \sum_{t=0}^TE[x_t\cdot v_t]\quad\forall x\in\X,\ v\in\V
\]
and
\[
\X_a = \X_0(\F_0)\times\cdots\times\X_T(\F_T).
\]

According to the general conjugate duality framework of \cite{roc74}, the {\em dual problem} is the concave maximization problem
% \begin{equation}\tag{$D_u$}
% \minimize\quad F^*(0,p,y)-\langle u,y\rangle\quad\ovr (p,y)\in\V\times\Y.
% \end{equation}
\begin{equation}\tag{$D$}\label{d}
\maximize\quad \langle\bar  u,y\rangle - F^*(0,p,y)\quad\ovr (p,y)\in\V\times\Y.
\end{equation}
More explicit forms will be given below. By definition, $\varphi^*(p,y)=F^*(0,p,y)$, so the dual problem can be written as
% \[
% \minimize\quad \varphi^*(p,y)-\langle u,y\rangle\quad\ovr (p,y)\in\V\times\Y.
% \]
\[
\maximize\quad \langle\bar  u,y\rangle - \varphi^*(p,y)\quad\ovr (p,y)\in\V\times\Y.
\]
% The dual value function is
% \[
% \gamma_u(v):=\inf_{(p,y)\in\V\times\Y} \{F^*(v,p,y)-\langle u,y\rangle\}.
% \]
By Fenchel's inequality,
\[
F(x,0,u) \ge \langle u,y\rangle-F^*(0,p,y)\quad\forall x\in\X,u\in\U,p\in\V,y\in\Y.
\]
Denoting the optimal values of primal and dual problem, respectively, as $\inf$\eqref{spx} and $\sup$\eqref{d}, we thus have
\[
\inf\eqref{spx} \ge\sup\eqref{d}
\]
A {\em duality gap} is said to exist if the inequality is strict. Conversely, we say that there is no duality gap if the above holds as an equality.

The associated {\em Lagrangian} is the convex-concave function $L$ on $\X\times\V\times\Y$ given by
\[
L(x,p,y):=\inf_{(z,u)\in\V\times\U}\{F(x,z,u)-\langle z,p\rangle-\langle u,y\rangle\}.
\]
By definition, the conjugate of $F$ can be expressed as
\[
F^*(v,p,y) = \sup_{x\in\X}\{\langle x,v\rangle - L(x,p,y)\}.
\]
The associated minimax problem is to find a saddle-value and/or a saddle-point of the concave-convex function
\[
L_{\bar u}(x,p,y) := L(x,p,y)+\langle\bar u,y\rangle,
\]
when minimizing over $x$ and maximizing over $(p,y)$. If
\[
\inf_x\sup_{p,y}L_{\bar u}(x,p,y) = \sup_{p,y}\inf_xL_{\bar u}(x,p,y),
\]
the common value is called the {\em minimax} or the {\em saddle-value}, and $(x,p,y)$ is called a {\em saddle-point} if 
\[
L_{\bar u}(x,p',y') \le L_{\bar u}(x,p,y) \le L_{\bar u}(x',p,y) \quad\forall x',p',y'.
\]
Existence of a saddle-point implies the existence of a saddle-value.

Since $\N$ is closed in probability, \thref{lem:rel} gives the following.

\begin{lemma}\thlabel{lem:Xaclosed}
$\X_a$ is $\sigma(\X,\V)$-closed.
\end{lemma}

The following three theorems are restatements of the main duality results in \cite{roc74} in the present setting. They all involve the assumption that the integral functional $Ef$ be closed in $u$. This means that $Ef(x,\cdot)$ is closed in $\U$ for each $x\in\X$. Combined with \thref{lem:Xaclosed}, this implies that the function $F$ is closed in $(z,u)$. 

The following characterizes the absence of duality gap.

\begin{theorem}\thlabel{thm:duality1}
The following are equivalent,
\begin{enumerate}
\item
  $\inf\eqref{spx}=\sup\eqref{d}$, %$\varphi(u)=-\gamma_u(v)$,
\item
  $\varphi$ is closed at $(0,\bar u)$. %$\varphi(u)=\cl\varphi(u)$.
\end{enumerate}
If $Ef$ is closed in $u$, the above are equivalent to
\begin{enumerate}[resume]
\item
  The function $L_{\bar u}$ has a saddle-value.
\end{enumerate}
%If $F$ is closed, the above are equivalent to
%\begin{enumerate}[resume]
%\item $\gamma_u$ is closed at $0$. %$\gamma_u(v)=\cl\gamma_u(v)$.
%\end{enumerate}
\end{theorem}

The next one characterizes situations where there is no duality gap and, furthermore, the dual admits solutions.

\begin{theorem}\thlabel{thm:duality2}
If $\varphi(0,u)<\infty$, the following are equivalent
\begin{enumerate}
  \item
    $(p,y)$ solves \eqref{d} and $\inf\eqref{spx}=\sup\eqref{d}$.
  \item
    either $\varphi(0,\bar u)=-\infty$ or $(p,y)\in\partial\varphi(0,\bar u)$,
\end{enumerate}
If $Ef$ is closed in $u$, the above are equivalent to
\begin{enumerate}[resume]
\item
  $\underset{x}\inf\,\underset{p,y}\sup\, L_{\bar u}(x,p,y) = \underset{x}\inf\, L_{\bar u}(x,p,y)$.
\end{enumerate}
\end{theorem}

%\begin{theorem}\thlabel{thm:duality2}
%Assume that $\varphi(0,u)<\infty$. Either $\varphi(0,u)=-\infty$ or $(p,y)\in\partial\varphi(0,u)$ if and only if there is no duality gap and $(p,y)$ solves the dual.
%\end{theorem}

The following characterizes the situations where both primal and dual solutions exist and there is no duality gap.

\begin{theorem}\thlabel{thm:duality3}
The following are equivalent,
\begin{enumerate}
\item
  $x$ solves \eqref{spx}, $(p,y)$ solves \eqref{d} and $\inf\eqref{spx}=\sup\eqref{d}\in\reals$,
%\item $(p,y)\in\partial\varphi(0,u)$ and $x$ solves the primal.
\item
  $(0,p,y)\in\partial F(x,0,\bar u)$.
\end{enumerate}
If $Ef$ is closed in $u$, the above are equivalent to
\begin{enumerate}[resume]
\item $0\in\partial_xL(x,p,y)$ and $(0,\bar u)\in\partial_{(p,y)}[-L](x,p,y)$.
\end{enumerate}
%If $F$ is closed, the above are equivalent to 
%\begin{enumerate}[resume]
%\item $x\in\partial\gamma_u(0)$ and $(p,y)$ solves the dual.
%\item $(x,0,u)\in\partial F^*(0,p,y)$.
%\end{enumerate}
\end{theorem}

In order to write the dual problem and the optimality conditions more explicitly in terms of the problem data, we will first derive explicit expressions for $F^*$ and $\varphi^*$. The rest of the section will then focus on the Lagrangian, the associated minimax problem and optimality conditions. We will denote the orthogonal complement of $\X_a$ by
\[
\X_a^\perp := \{v\in\V\,|\,\langle x,v\rangle = 0\quad \forall x\in\X_a\}.
\]

\begin{lemma}\thlabel{lem:dual}
If $\dom Ef\cap(\X\times\U)\ne\emptyset$, then
\begin{align*}
F^*(v,p,y) &= Ef^*(v+p,y)+\delta_{\X_a^\perp}(p),
\end{align*}
and, in particular, 
\[
\varphi^*(p,y)= Ef^*(p,y)+\delta_{\X_a^\perp}(p).
\]
If, in addition, $\dom Ef^*\cap(\V\times\Y)\ne\emptyset$, then $F$ is proper and closed.
\end{lemma}

\begin{proof}
By the interchange rule \cite[Theorem~14.60]{rw98},
\begin{align*}
F^*(v,p,y) &= \sup_{x\in\X,z\in\X,u\in\U}\{\langle x,v\rangle + \langle z,p\rangle + \langle u,y\rangle - Ef(x,u) \,|\, x-z\in\X_a\}\\
&= \sup_{x\in \X,z'\in\X,u\in\U}\{E[x\cdot(v+p) + u\cdot y - f(x,u) - z'\cdot p] \,|\, z'\in\X_a\}\\
&= Ef^*(v+p,y)+\delta_{\X_a^\perp}(p).
\end{align*}
When $\dom Ef^*\ne\emptyset$, $Ef$ is proper and closed, by \thref{cor:cif}, so $F$ is closed as a sum of proper and closed functions. Clearly, $F$ is proper.
\end{proof}

As an immediate corollary, we get the following.

\begin{theorem}\thlabel{thm:dualproblem}
If $\dom Ef\cap(\X\times\U)\ne\emptyset$, the dual problem \eqref{d} can be written as
\begin{equation}\label{d}\tag{$D$}
\maximize\quad \langle\bar u,y\rangle - Ef^*(p,y)\quad\ovr\quad (p,y)\in\X_a^\perp\times\Y.
\end{equation}
\end{theorem}

Much as in \cite[Section~4]{rw76} and \cite[Proposition~1]{pli82}, the orthogonal complement of $\X_a$ can be expressed in terms of the set
\[
\N^\perp:=\{v\in L^1\mid\langle x,v\rangle=0\quad \forall x\in\N\cap L^\infty\}.
\]

\begin{lemma}\thlabel{lem:vperp}
The set $\X_a$ is $\sigma(\X,\V)$-closed and
%\[
%\X_a = \{x\in\X\mid \langle x,v\rangle = 0\quad\forall v\in\X_a^\perp\}.
%\]
%If $\X$ and $\V$ are solid [this has been already made a standing assumption??], then
\[
\X_a^\perp = \N^\perp\cap\V = \{v\in\V\,|\,E_tv_t=0\quad t=0,\ldots,T\}.
\]
\end{lemma}

\begin{proof}
Since $\N$ $\N$ is closed in $L^0$, Lemma~\ref{lem:rel} implies that $\X_a$ is closed in $\tau(\X,\V)$ and thus, by convexity, also in $\sigma(\X,\V)$. %The first equality thus follows from the bipolar theorem (see \thref{cor:bpt}).
Since
\[
\X_a = \X_0(\F_0)\times\cdots\times\X_T(\F_T),
\]
we have $v\in\X_a^\perp$ if and only if $E[x_t\cdot v_t]=0$ for every $x_t\in\X_t(\F_t)$. Here, $E[x_t\cdot v_t]=E[x_t\cdot (E_t v_t)]$, by \thref{lem:ce}.
\end{proof}

Note that the dual objective can be written also as
\[
\langle\bar u,y\rangle - Ef^*(p,y) = E\inf_{(x,u)\in\reals^n\times\reals^m}[f(x,u)-x\cdot p+(\bar u-u)\cdot y].
\]
This is the optimum value in a relaxed version of the primal problem \eqref{sp} where we are now allowed to optimize over both $x$ and $u$ and the information constraint $x\in\N$ has been removed so the minimization can be done. The constraints have been replaced by linear penalties given by the dual variables $p$ and $y$. The optimum value of \eqref{d} is less than or equal to that of \eqref{spx}. If the value function $\varphi$ is closed at $(0,\bar u)$, then, by \thref{thm:duality1}, the optimum values can be made arbitrarily close by an appropriate choice of $(p,y)$. If $(p,y)\in\partial \varphi(0,\bar u)$, then, by \thref{thm:duality2}, there is no duality gap and $(p,y)$ solves the dual. This implies, in particular, that $p$ is a subgradient of $\varphi$ with respect to the first argument at $(0,\bar u)$, i.e.,
\[
Ef(x+z,\bar u)-\langle z,p\rangle \ge \varphi(0,\bar u)\quad\forall x\in\X_a,z\in\X.
\]
In other words, one cannot improve the optimum value of \eqref{spx} by adding a nonadapted perturbation $z$ to the strategy $x$ when one has to pay $\langle z,p\rangle$. Such an element $p\in\X_a^\perp$ is known as a {\em shadow price of information} of \eqref{spx}. In the deterministic setting, $\X_a^\perp=\{0\}$ so the dual problem becomes
\[
\maximize\quad\bar u\cdot y -f^*(0,y)\quad\ovr\quad y \in\reals^m
\]
and we recover the classical duality framework in finite-dimensional spaces.

\thref{thm:dualproblem} can be used to restate \thref{thm:duality1,thm:duality2,thm:duality3} more explicitly. In particular, the first two equivalences in \thref{thm:duality3} can be restated as follows.

\begin{theorem}\thlabel{thm:duality}
If \eqref{spx} and \eqref{d} are feasible, then the following are equivalent
\begin{enumerate}
\item\label{kkt:1}
$x$ solves \eqref{spx}, $(p,y)$ solves \eqref{d} and $\inf\eqref{spx}=\sup\eqref{d}$,
\item\label{kkt:2}
$x\in\X_a$, $(p,y)\in\X_a^\perp\times\Y$ and
\[
(p,y)\in\partial f(x,\bar u)\quad \text{a.s.}
\]
\end{enumerate}
\end{theorem}

\begin{proof}
By \thref{thm:duality3}, \ref{kkt:1} is equivalent to $(0,p,y)\in\partial F(x,0,\bar u)$ which means that $F(x,0,\bar u)+F^*(0,p,y)=\langle\bar u,y\rangle$. By Lemma~\ref{lem:dual}, this means that $x\in\X_a$, $p\in\X_a^\perp$ and
\begin{equation}\label{eq:kkt0}
Ef(x,\bar u)+Ef^*(p,y)=E[x\cdot p]+E[\bar u\cdot y].
\end{equation}
Given $(x',u')\in\X\times\U$ and $(p',y')\in\V\times\Y$, we have
\begin{equation}\label{eq:kkt0b}
f(x',u')+f^*(p',y')\ge x'\cdot p' + u\cdot y',
\end{equation}
by Fenchel's inequality, so the feasibility assumptions imply that the negative parts of $f(x',u')$ and $f^*(p',y')$ are integrable and thus, by \thref{lem:Esum},
\[
Ef(x',u')+Ef^*(p',y') = E[f(x',u')+f^*(p',y')].
\]
Thus, \eqref{eq:kkt0} means that $(x,\bar u)$ and $(p,y)$ satisfy \eqref{eq:kkt0b} as an equality, i.e.,  $(p,y)\in\partial f(x,\bar u)$. %Equivalence of \ref{kkt:2} and \ref{kkt:3} follows from scenariowise application of \thref{lem:sgvarphi}.
\end{proof}

Note that the dual is feasible e.g.\ if $F$ is bounded from below since then $F^*(0,0)$ is finite. If $\partial\varphi(0,\bar u)$ is nonempty, then by \thref{thm:duality2}, there is no duality gap and a dual has a solution. \thref{thm:duality} thus implies the following.

\begin{corollary}\thlabel{cor:kkt0}
If $\partial\varphi(0,\bar u)\ne\emptyset$, then $\inf\eqref{spx}=\sup\eqref{d}$ and the following are equivalent for an $x\in\X_a$,
\begin{enumerate}
\item\label{kkt:1}
$x$ solves \eqref{spx},
\item\label{kkt:2}
there exists $(p,y)\in\X_a^\perp\times\Y$ with 
\[
(p,y)\in\partial f(x,\bar u)\quad\text{a.s.}
\]
\end{enumerate}
\end{corollary}

The rest of the section focuses on the Lagrangian $L$ and the associated minimax problem. The Lagrangian $L$ itself has a somewhat cumbersome expression but it turns out that it is ``equivalent'' to a simpler function that has the same saddle-value and saddle-points. The expressions derived below, involve the 
%By definiton, the function $-L$ is closed in $(p,y)\in\V\times\Y$ but $L$ may fail to be lower semicontinuous in $x\in\X$. By the biconjugate theorem,
%\[
%(\cl_xL)(x,p,y) = \sup_{v\in\V}\{\langle x,v\rangle - F^*(v,p,y)\}.
%\]
%This function is closed in $x$ but it may fail to be upper semicontinuous in $(p,y)$. Similarly, the Lagrangian integrand $l$ may fail to be closed in $x\in\reals^n$ while the function
%\[
%(\cl_xl)(x,y,\omega) = \sup_{v\in\reals^n}\{x\cdot v - f^*(v,y,\omega)\} 
%\]
%may fail to be usc in $y\in\reals^m$.
%Recall that a saddle-function $\tilde L$ belongs to the equivalence class of $L$ if it lies between $L$ and $\cl_x L$, where, given $(p,y)$, 
%\[
%(\cl_x L)(x,p,y) :=\cl_x L(\cdot,p,y)(x).
%\]
%By ?? in the appendix, $\tilde L$ and $L$ have the same saddle-value and the same saddle-points if any, $\inf_x \tilde L(x,p,y)=\inf L(x,p,y)$ and, as soon as $F$ is closed, $\sup_{p,y} \tilde L(x,p,y)=\sup_{p,y} L(x,p,y)$. Thus the Lagrangian can be replaced by any saddle-function from the equivalence class of $L$. 
{\em Lagrangian integrand} $l:\reals^n\times\reals^m\times\Omega\to\ereals$ defined by
\[
l(x,y,\omega) := \inf_{u\in\reals^m}\{f(x,u,\omega)-u\cdot y\}.
\]
For any $(x,y,\omega)$, the function $l(\cdot,y,\omega)$ is convex and $l(x,\cdot,\omega)$ is upper semicontinuous and concave. Clearly,
\[
f^*(v,y,\omega) = \sup_{x\in\reals^n}\{x\cdot v - l(x,y,\omega)\}
\]
so, by the biconjugate theorem,
\begin{align*}
(\cl_x l)(x,y,\omega) = \sup_{v\in\reals^n}\{x\cdot v - f^*(v,y,\omega)\}.
\end{align*}
Given $x\in\X$ and $y\in\Y$, the functions
\[
(y,\omega) \mapsto-l(x(\omega), y,\omega) = \sup_{u\in\reals^m}\{u\cdot y-f(x(\omega),u,\omega)\}
\]
and
\[
(x,\omega) \mapsto (\cl_xl)(x, y(\omega),\omega) = \sup_{v\in\reals^n}\{x\cdot v-f^*(v,y(\omega),\omega)\}
\]
are normal integrands, by Proposition~14.45 and Theorem~14.50 of \cite{rw98}. In particular, the functions
\[
\omega\mapsto l(x(\omega),y(\omega),\omega)\quad\text{and}\quad\omega\mapsto(\cl_xl)(x(\omega),y(\omega),\omega)
\]
are measurable, by \cite[Proposition~14.28]{rw98},
%Note that
%\[
%(x,\omega) \mapsto l(x, y(\omega),\omega)
%\]
%is not necessarily a normal integrand, because it may fail to be lsc in $x$.

We will denote the projection of $\dom Ef$ to the $x$ component by 
\[
\dom_xEf := \{x\in\X\mid \exists u\in\U:\ Ef(x,u)<\infty\}
\]
and the projection of $\dom Ef^*$ to the $y$ component by
\[
\dom_yEf^* := \{y\in\Y\mid \exists v\in\U:\ Ef^*(v,y)<\infty\}.
\]

\begin{lemma}\thlabel{lem:L}
We have
\begin{align*}
  L(x,p,y) &=
  \begin{cases}
  +\infty & \text{if $x\notin\dom_xEf$},\\
  El(x,y) - \langle x,p\rangle & \text{if $x\in\dom_xEf$ and $p\in\X_a^\perp$},\\
  -\infty & \text{otherwise}.
  \end{cases}
\end{align*}
If $\dom Ef\ne\emptyset$, then 
\begin{align*}
  (\cl_xL)(x,p,y)   &=
    \begin{cases}
    E(\cl_x l)(x,y) - \langle x,p\rangle & \text{if $y\in\dom_yEf^*$ and $p\in\X_a^\perp$},\\
    -\infty & \text{otherwise}.
    \end{cases}
\end{align*}
If $\dom Ef\ne\emptyset$ and $\dom Ef^*\ne\emptyset$, then all convex-concave functions between $L$ and $\cl_x L$ have the same saddle-value, saddle-points and subdifferentials. In this case, 
\[
v\in\partial_x L(x,p,y),\quad (z,u)\in\partial_{p,y}[-L](x,p,y)
\]
if and only if $x-z\in\X_a$, $p\in\X_a^\perp$ and 
\begin{align*}
p+v\in\partial_x l(x,y),\quad u\in\partial_y [-l](x,y) \quad \text{a.s.}
\end{align*}
% In particular, the convex-concave functions
% \[
% \hat L(x,p,y) :=\begin{cases}
% El(x,y) - \langle x,p\rangle & \text{if $p\in\N^\perp$},\\
% -\infty & \text{otherwise},
% \end{cases}
% \]
% and
% \[
% \check L(x,p,y) :=\begin{cases}
% E(\cl_x l)(x,y) - \langle x,p\rangle & \text{if $p\in\N^\perp$},\\
% -\infty & \text{otherwise},
% \end{cases}
% \]
% are such.
\end{lemma}

\begin{proof}
By definition,
\begin{align*}
L(x,p,y) &= \inf_{(z,u)\in\X\times\U}\{F(x,z,u) - \langle z,p\rangle - \langle u,y\rangle\}\\
&= \inf_{(z,u)\in\X\times\U}\{E[f(x,u)-z\cdot p-u\cdot y] \,|\, x-z\in\X_a\}\\
&= \inf_{(z',u)\in\X\times\U}\{E[f(x,u)-(x-z')\cdot p-u\cdot y] \,|\, z'\in\X_a\},
\end{align*}
so the expression for $L$ follows from \cite[Theorem 14.60]{rw98}. By \thref{lem:dual},
\begin{align*}
  (\cl_xL)(x,p,y) &=\sup_{v\in\V}\{\langle x,v\rangle - F^*(v,p,y)\}\\
  &= 
  \begin{cases}
    \sup_{v\in\V}\{\langle x,v\rangle - Ef^*(v+p,y)\} & \text{if $p\in\X_a^\perp$},\\
    -\infty & \text{otherwise}
  \end{cases}
\end{align*}
so the expression for $\cl_x L$ follows from \cite[Theorem 14.60]{rw98} again. When $\dom Ef\ne\emptyset$ and $\dom Ef^*\ne\emptyset$, the function $F$ is proper and closed by \thref{lem:dual}, so the saddle-values, saddle-points and subdifferentials of $L$ and $\cl_x L$ coincide, by \cite[Theorem~2 and 7]{roc71c}.

When $\dom Ef\ne\emptyset$ and $\dom Ef^*\ne\emptyset$, $F$ is closed, by \thref{cor:cif}, and then, $(v,z,u)\in\partial L(x,p,y)$ if and only if $(v,p,y)\in\partial F(x,z,u)$. By \thref{lem:dual}, this means that $x-z\in\X_a$, $p\in\X_a^\perp$ and
\[
Ef(x,u) + Ef^*(v+p,y) = E[x\cdot v] + E[z\cdot p] + E[u\cdot y]
\]
or, equivalently,
\[
Ef(x,u) + Ef^*(v+p,y) = E[x\cdot(v+p)] + E[u\cdot y]
\]
Since, by Fenchel's inequality,
\[
f(x,u,\omega)+f^*(v+p,y,\omega)\ge x\cdot(v+p)+u\cdot y,
\]
this means that $(v+p,y)\in\partial f(x,u)$ almost surely. Since $f$ is closed, this is equivalent to $v+p\in\partial_x l(x,y)$ and $u\in\partial_y[-l](x,y)$.
\end{proof}

\begin{corollary}\thlabel{cor:kkt}
If \eqref{spx} and \eqref{d} are feasible, the following are equivalent with the conditions in \thref{thm:duality},
\begin{enumerate}
\item
$(x,p,y)$ is a saddle-point of 
\[
(x,p,y)\mapsto El(x,y)-\langle x,p\rangle+\langle\bar u,y\rangle,
\]
when minimizing over $x\in\X$ and maximizing over $(p,y)\in\X_a^\perp\times\Y$,
\item\label{kkt:3}
$x\in\X_a$, $(p,y)\in\X_a^\perp\times\Y$ and
\begin{align*}
p\in\partial_x l(x,y),\quad \bar u\in\partial_y[-l](x,y)\quad \text{a.s.}
\end{align*}
\end{enumerate}
\end{corollary}
\begin{proof}
The claim follows from \thref{thm:duality3,lem:L}, since the convex-concave function in the first condition lies between $L$ and $\cl_x L$. 
\end{proof}

Similarly, we can augment \thref{cor:kkt0} as follows.

\begin{corollary}\thlabel{cor:kkt2}
If $\partial\varphi(0,\bar u)\ne\emptyset$, then the following are equivalent,
\begin{enumerate}
\item\label{kkt:1}
$x$ solves \eqref{spx},
\item\label{kkt:2}
  $x\in\X_a$ and there exists $(p,y)\in\X_a^\perp\times\Y$ with 
\[
(p,y)\in\partial f(x,\bar u)\quad\text{a.s.}
\]
\item
  $x\in\X_a$ and there exists $(p,y)\in\X_a^\perp\times\Y$ with
  \[
  p\in\partial_x l(x,y),\quad \bar u\in\partial_y[-l](x,y)\quad \text{a.s.}
  \]
\end{enumerate}
\end{corollary}

In the deterministic setting, $\X_a^\perp=\{0\}$ so condition 3 in \thref{cor:kkt2} becomes the Karush-Kuhn-Tucker (KKT) condition in finite-dimensional convex optimization. In the stochastic setting, the shadow price of information $p\in\X_a^\perp$ allows us to write the KKT-conditions scenariowise.

% --the below implies the above: adding the $p$-dimension and the indicator of $\N^\perp$ gives another closed saddle function. Adding the continuous bilinear term preserves closedness??

% \begin{remark}
% Define the saddle-function $(\cl_yEl)$ on $L^\infty\times\Y$ by
% \begin{align*}
% (\cl_yEl)(x,y) &:= \inf_{u}\{Ef(x,u)-\langle u,y\rangle\}\\
% &=\begin{cases}
%   +\infty & \text{if $x\notin\dom_xEf$},\\
%   E[l(x,y)] & \text{otherwise}.
% \end{cases}
% \end{align*}
% The notation $\cl_yEl$ comes from the fact that, as soon as $Ef$ is proper and closed, $\cl_yEl$ is the upper closure of the saddle-function
% \[
% (El)(x,y):=El(x,y).
% \]
% This follows just like in Remark~\ref{rem:equivalence}. Similarly, the lower closure of $El$ is given by
% \begin{align*}
% (\cl_xEl)(x,y) &:= \sup_{v}\{\langle x,v\rangle-Ef^*(v,y)\}\\
% &=\begin{cases}
%   E[(\cl_x l)(x,y)] & \text{if $y\in\dom_yEf^*$},\\
%   -\infty & \text{otherwise}.
% \end{cases}
% \end{align*}
% The equivalence class contains also the saddle-function $(x,y)\mapsto E(\cl_x l)(x,y)$.
% \end{remark}

\section{Duality for $(SP)$}\label{sec:rel}

While problem \eqref{spx} in Section~\ref{sec:dual} allows for a convenient dualization within the purely functional analytic conjugate duality framework, there are interesting applications where $\inf\eqref{spx}>\inf\eqref{sp}$ or the infimum in \eqref{sp} is attained in $L^0$ but not in $\X$; see \thref{ex:ceN} for a simple illustration. It may even happen that \eqref{spx} is infeasible while \eqref{sp} is not; see \thref{ex:spxna}. This section shows that many of the duality relations between \eqref{spx} and \eqref{d} derived in Section~\ref{sec:dual} also hold between \eqref{sp} and \eqref{d}.

The function
\[
\bar\varphi(z,u):=\inf_{x\in L^0}\{Ef(x,u)\mid x-z\in\N\}
\]
on $\X\times\U$ gives the optimum value of \eqref{sp} when we perturb the strategies $x$ by $z\in\X$ and vary the parameter $\bar u$ in the space $\U$. In particular, $\bar\varphi(0,\bar u)=\inf\eqref{sp}$. Clearly, $\varphi\ge\bar\varphi$ since the latter is defined by optimizing over a larger class of strategies. However, under a mild condition, their conjugates coincide. 

\begin{lemma}\thlabel{varphibar}
If $\dom Ef\cap(\X\times\U)\ne\emptyset$, then $\varphi^*=\bar\varphi^*$. %In particular, $\cl\varphi=\cl \bar\varphi$. We have $\partial\varphi(z,u)=\partial \bar\varphi(p,y)$ whenever the left side is nonempty.
\end{lemma}

\begin{proof}
Since $\varphi\ge\bar\varphi$, we have $\varphi^*\le \bar\varphi^*$. To prove the converse, let $(p,y)\in\dom\varphi^*$. By \thref{lem:dual},
\[
\varphi^*(p,y)=Ef^*(p,y)+\delta_{\X_a^\perp}(p),
\]
so $p\in\X_a^\perp$. By Fenchel's inequality,
\begin{align*}
Ef(x,u)+\delta_\N(x-z)+Ef^*(p,y) \ge E[(x-z)\cdot p] + E[z\cdot p] + E[u \cdot y]
\end{align*}
for all $(x,z,u)\in L^0\times\X\times\U$, so \thref{lem:perp} below implies
\begin{align*}
Ef(x,u)+\delta_\N(x-z)+Ef^*(p,y) \ge E[z\cdot p] + E[u \cdot y].
\end{align*}
Thus $\bar\varphi(z,u)+\varphi^*(p,y) \ge \langle z,p\rangle+\langle u,y\rangle$
 for all $(z,u)\in\X\times \U$, which means that $\bar\varphi^*(p,y)\le \varphi^*(p,y)$.
\end{proof}

The above proof used the following from \cite{per16}; see also \cite{pp22}.
\begin{lemma}\thlabel{lem:perp}
Let $x\in\N$ and $v\in\N^\perp$. If $E[x\cdot v]^+\in L^1$, then $E[x\cdot v]=0$.
\end{lemma}

Note that, if $\dom Ef\cap(\X\times\U)\ne\emptyset$ and $\varphi^*$ is proper, then $Ef$ is proper on $\X\times\U$.

\begin{corollary}\thlabel{varphibar2}
We have $\partial\varphi(z,u)=\partial \bar\varphi(z,u)$ whenever the left side is nonempty. In particular, if $\partial\varphi(0,\bar u)\ne\emptyset$, then
\[
\inf\eqref{sp}=\inf\eqref{spx}=\sup\eqref{d}
\]
and the dual optimum is attained.
\end{corollary}

\begin{proof}
If $\partial\varphi(z,u)\ne\emptyset$, we have $(z,u)\in\dom\varphi$ and thus, $\dom Ef\cap(\X\times\U)\ne\emptyset$. By the biconjugate theorem and \thref{varphibar}, $\cl\varphi=\cl\bar\varphi$. In particular, $\varphi\ge\bar\varphi\ge\cl\varphi$. When $\partial\varphi(z,u)\ne\emptyset$, we have $\varphi(z,u)=\cl\varphi(z,u)$ so $\bar\varphi(z,u)=\varphi(z,u)$ and thus, $\bar\varphi(z,u)+\bar\varphi^*(p,y)=\langle x,p\rangle+\langle u,y\rangle$ if and only if $\varphi(z,u)+\varphi^*(p,y)=\langle x,p\rangle+\langle u,y\rangle$. In other words, $(p,y)\in\partial\bar\varphi(z,u)$ if and only if $(p,y)\in\partial\varphi(z,u)$. The second claim follows from \thref{thm:duality2} and \thref{varphibar} since subdifferentiability implies closedness.
\end{proof}
By \thref{varphibar},
\[
\bar\varphi(0,\bar u)\ge\langle \bar u,y\rangle-\varphi^*(p,y)\quad\forall (p,y)\in\V\times\Y,
\]
where the right side is the dual objective from Section~\ref{sec:dual}. Thus, the optimal value of \eqref{sp} is always bounded from below by the dual objective so the duality gap between \eqref{sp}  and \eqref{d} is nonnegative. The duality gap is zero if and only if $\bar\varphi^{**}(0,\bar u)=\bar\varphi(0,\bar u)$. Thus, we have the following, which gives the analogue of the first equivalence in \thref{thm:duality1}.

\begin{theorem}\thlabel{thm:relduality1}
If $\dom Ef\cap(\X\times\U)\ne\emptyset$, then the following are equivalent,
\begin{enumerate}
\item
  $\inf\eqref{sp}=\sup\eqref{d}$,
\item
  $\bar \varphi$ is closed at $(0,\bar u)$. %$\varphi(u)=\cl\varphi(u)$.
\end{enumerate}
%If $F$ is closed, the above are equivalent to
%\begin{enumerate}[resume]
%\item $\gamma_u$ is closed at $0$. %$\gamma_u(v)=\cl\gamma_u(v)$.
%\end{enumerate}
\end{theorem}

The following gives the analogue of the first equivalence in \thref{thm:duality2}.

\begin{theorem}\thlabel{thm:dualityrel2}
If $\dom Ef\cap(\X\times\U)\ne\emptyset$ and $\bar\varphi(0,u)<\infty$, then the following are equivalent
\begin{enumerate}
  \item
    $(p,y)$ solves \eqref{d} and $\inf\eqref{sp}=\sup\eqref{d}$,
  \item
    either $\bar\varphi(0,\bar u)=-\infty$ or $(p,y)\in\partial\bar\varphi(0,\bar u)$,
\end{enumerate}
% If $F$ is closed in $u$, the above are equivalent to
% \begin{enumerate}[resume]
% \item
%   $\underset{x}\inf\,\underset{p,y}\sup\, L_u(x,p,y) = \underset{x}\inf\, L_u(x,p,y)$.
% \end{enumerate}
\end{theorem}

\begin{proof}
Condition 2 means that either $\bar\varphi(0,\bar u)=-\infty$  or $\bar\varphi(0,z)+\bar\varphi^*(p,y)=\langle\bar  u,y\rangle$, where, by \thref{varphibar}, $\bar\varphi^*=\varphi^*$. Thus the claim follows from \thref{lem:dual}. 
\end{proof}

\begin{theorem}\thlabel{thm:kktr}
If $\dom Ef\cap(\X\times\U)\ne\emptyset$ and \eqref{sp} and \eqref{d} are feasible, then following are equivalent
\begin{enumerate}
\item
$x$ solves \eqref{sp}, $(p,y)$ solves \eqref{d} and $\inf\eqref{sp}=\sup\eqref{d}$,
\item
$x$ is feasible in \eqref{sp}, $(p,y)$ is feasible in \eqref{d} and
\begin{equation}\label{eq:kkt}
(p,y)\in\partial f(x,\bar u)\quad P\text{-a.s.}
\end{equation}
\item
$x$ is feasible in \eqref{sp}, $(p,y)$ is feasible in \eqref{d}  and
\begin{align*}
p\in\partial_x l(x,y),\quad \bar u\in\partial_y[-l](x,y)\quad P\text{-a.s.}
\end{align*}
\end{enumerate}
\end{theorem}

\begin{proof}
The assumptions imply that $f$ is proper so the equivalence of 2 and 3 follows from \cite[Theorem~37.5]{roc70a}. Let $x\in \N$ and $(p,y)\in \V\times\Y$ be feasible. By Fenchel's inequality,
\[
f(x,u)+f^*(p,y)-\bar u\cdot y \ge x\cdot p\quad P\text{-a.s.}
\]
so 
\[
Ef(x,u)+E[f^*(p,y)-\bar u\cdot y]\ge E[x\cdot p]
\]
and one holds as an equality if and only if the other one does. Equality in the former means that 2 holds. By Lemma~\ref{lem:perp}, $E[x\cdot p]=0$, so equality in the latter means that 1 holds.
\end{proof}

If $\partial\bar \varphi(0,\bar u)\ne\emptyset$, then, by \thref{thm:dualityrel2}, $\inf\eqref{sp}=\sup\eqref{d}$ and the dual has a solution. \thref{thm:kktr} thus gives the following.

\begin{corollary}\thlabel{cor:kkt4}
If $\dom Ef\cap(\X\times\U)\ne\emptyset$ and $\partial\bar\varphi(0,\bar u)\ne\emptyset$, then $\inf\eqref{sp}=\sup\eqref{d}$, the optimum is attained in the dual and the following are equivalent,
\begin{enumerate}
\item\label{kkt:1}
$x$ solves \eqref{sp},
\item\label{kkt:2}
$x$ is feasible and there exists a dual feasible $(p,y)$ with $(p,y)\in\partial f(x,\bar u)$ almost surely.
\end{enumerate}
\end{corollary}

If \eqref{sp} has a solution, $(p,y)$ is dual optimal and $\inf\eqref{sp}=\sup\eqref{d}$, then, by \thref{thm:kktr}, solutions of \eqref{sp} are scenariowise minimizers of the function 
\[
x\mapsto l(x,y)-x\cdot p
\]
and, in particular, if the scenariowise minimizer is unique, then it is necessarily adapted and solves \eqref{sp}.

\begin{cexample}\thlabel{ex:ceN}
It is possible that $\inf\eqref{sp}=\sup\eqref{d}$ while $\inf\eqref{spx}>\sup\eqref{d}$. Indeed, let
\[
f(x,u,\omega)=(x_0-1)^2+\delta_{\{0\}}(x_0\xi(\omega)-x_1),
\]
$\F_0$ be trivial and $\xi\in L^0(\F_1)$ with $\xi\notin \X$. Since $f$ is nonnegative, $(1,\xi)$ is optimal for \eqref{sp} and the optimal value is zero. Here $Ef$ is proper on $\X\times\U$, and, by a direct verification, $f^*(0,0)=0$, so the origin is a dual solution and $\inf\eqref{sp}=\sup\eqref{d}=0$. On the other hand, the only feasible solution of \eqref{spx} is the origin, so   $\inf\eqref{spx}=1$.
\end{cexample}

\begin{cexample}\thlabel{ex:spxna}
It may happen that \eqref{spx} is infeasible, but nevertheless, \eqref{sp} is feasible and $\inf\eqref{sp}=\sup\eqref{d}$. Indeed, let
\[
f(x,u,\omega)=\delta_{\{0\}}(x_T-u\xi(\omega)).
\]
If $\xi\notin\X$ and $\bar u=1$, then \eqref{sp} is feasible while \eqref{spx} is not. Clearly $Ef$ is proper on $\X\times\U$ and $f^*(0,0)=0$, so $\inf\eqref{sp}=\sup\eqref{d}=0$.
\end{cexample}

The optimum value of the dual problem
\begin{equation}\tag{$D$}
\maximize\quad \langle u,y\rangle - \varphi^*(p,y) \quad\over (p,y)\in\V\times\Y
\end{equation}
clearly coincides with that of
\begin{equation}\label{rd}\tag{$rD$}
\maximize\quad \langle u,y\rangle - g(y) \quad\over (p,y)\in\V\times\Y,
\end{equation}
where
\[
g(y):=\inf_{p\in\V} \varphi^*(p,y).
\]
Problem \eqref{rd} is called the {\em reduced dual problem}. A pair $(p,y)$ solves \eqref{d} if and only if $y$ solves \eqref{rd} and $p$ attains the infimum in the definition of $g$. In many applications, the infimum and the minimizing $p$ can be found analytically.

\section{Absence of a duality gap}\label{sec:adg}

This section recalls the main result of \cite{per16} on the lower semicontinuity of the optimum value function of \eqref{sp}. As we have seen, the lower semicontinuity implies the absence of a duality gap. Besides the lower semicontinuity, \thref{thm:varphi} below establishes the existence of optimal solutions to \eqref{sp}.

\begin{assumption}\thlabel{ass:adg}
\eqref{sp} is feasible, 
\[
\{x\in\N\mid f^\infty(x,0)\le 0\}
\]
is a linear space and there exists $p\in\X_a^\perp$ and $\epsilon>0$ such that
\[
\inf_{y\in\Y} Ef^*(\lambda p,y)<\infty
\]
for $\lambda\in[1-\epsilon,1+\epsilon]$.
\end{assumption}

The linearity condition in \thref{ass:adg} holds trivially if $f(\cdot,0)$ is inf-compact since then, its recession function is strictly positive except at the origin; see \cite[Theorem~8.6]{roc70a}. If $f$ is bounded from below by an integrable random variable, then $Ef^*(0,0)<\infty$ so the second condition in \thref{ass:adg} holds. The second condition holds also e.g.\ if $\dom Ef^*\cap(\V\times\Y)$ is a nonempty cone. In certain models of financial mathematics, it is implied by the well-known asymptotic elasticity conditions on the utility function; see \cite[Section~5.5]{pp22}.

%By \thref{rem:homif}, this holds if $Ef^*$ is proper on $\V\times\Y$ and $h$ is $p$-homogeneous with $p\ne 1$; see \thref{lem:homconj}. The following gives a primal formulation of the second condition in \thref{lem:2lambda}.

\begin{theorem}\thlabel{thm:varphi}
Under \thref{ass:adg}, the function
\[
\bar\varphi(z,u) = \inf_{x\in L^0}\{Ef(x,u)\mid x-z\in\N\}
\]
is lower semicontinuous on $\X\times \U$,
\[
\bar\varphi^\infty(z,u) = \inf_{x\in L^0}E \{f^\infty(x,u)\mid z-z\in\N\},
\]
and the infimums are attained for every $(z,u)\in\X\times\U$.
\end{theorem}

\begin{proof}
We have
\[
\bar\varphi(z,u) = \inf_{x\in\N}Ef(x+z,u)=\inf_{x\in\N}E\bar f(x,z,u),
\]
where $\bar f(x,z,u,\omega)=f(x+z,u,\omega)$. The claim thus follows from the main result of \cite{per16} as soon as
\[
\{x\in\N\mid \bar f^\infty(x,0,0)\le 0\} 
\]
is linear and there exists $p\in \X^\perp_a$ and $\epsilon>0$
\[
\inf_{(p',y)\in\V\times\Y} E\bar f^*(\lambda p, p',y)<\infty.
\]
Since $\bar f^\infty(x,z,u,\omega)=f^\infty(x+z,u,\omega)$, the former is clear from the linearity condition in \thref{ass:adg}.  We have 
\[
\bar f^*(v,p,y,\omega)=f^*(v,y,\omega)+\delta_{\{0\}}(p-v)
\]
so the latter follows from \thref{ass:adg} as well.
\end{proof}

Combining the above with \thref{thm:relduality1} gives the following.

\begin{corollary}
Under \thref{ass:adg}, $\inf\eqref{sp}=\sup\eqref{d}$, and \eqref{sp} has a solution.
\end{corollary}

\section{Applications}\label{sec:appdual}

This section applies the general duality result to specific instances of \eqref{sp}. In the following applications, we give more explicit expressions for the involved functions and conditions but only give selected statements as examples of how the general theory can be applied.

\subsection{Mathematical programming}\label{sec:mp3}

Consider the problem
\begin{equation}\label{mp}\tag{$MP$}
\begin{aligned}
&\minimize\quad & Ef_0(x)&\quad\ovr\ x\in\N,\\
  &\st\quad & f_j(x) &\le 0\quad j=1,\ldots,l\ a.s.,\\
   & & f_j(x) &= 0\quad j=l+1,\ldots,m\ a.s.
\end{aligned}
\end{equation}
where $f_j$ are convex normal integrands with $f_j$ affine for $j>l$. This fits the general duality framework with $\bar u=0$ and
%\[
%f(x,u,\omega) = 
%\begin{cases}
%  f_0(x,\omega) & \text{if $f_j(x,\omega)+u_j\le 0,\ j=1,\ldots,l$, $f_j(x,\omega)+u_j= 0,\ j=l+1,\ldots,m$},\\
%  +\infty & \text{otherwise}.
%\end{cases}
%\]
%The Lagrandian integrand becomes
%\[
%l(x,y,\omega)=
%\begin{cases}
%+\infty & \text{if $x\notin\cap_{j=0}^m\dom f_j(\cdot,\omega)$},\\
%f_0(x,\omega) + \sum_{j=1}^my_jf_j(x,\omega) & \text{if $x\in\cap_{j=0}^m\dom f_j(\cdot,\omega)$ and $y_j\ge 0$ for $j>l$},\\
%-\infty & \text{otherwise}.
%\end{cases}
%\]
%In general, we don't have an explicit expression for $f^*$; see however Example~\ref{??} below. The optimality conditions
%\[
%p\in\partial_x l(x,y),\quad u\in\partial_y[-l](x,y)\quad \text{a.s.}
%\]
%can be written as
%\begin{gather*}
%  p\in\partial_x[f_0(x) + \sum_{j=1}^my_jf_j(x)]\quad \text{a.s.},\\
%  f_j(x) + u_j \le 0,\quad y_j\ge 0,\quad y_j[f_j(x) + u_j]=0\quad j=1,\ldots,l\quad \text{a.s.},\\
%  f_j(x) + u_j = 0\quad j=l+1,\ldots,m\quad \text{a.s.}  
%\end{gather*}
\[
f(x,u,\omega) = 
\begin{cases}
  f_0(x,\omega) & \text{if $x\in\dom H,\ H(x)+u\in K$},\\
  +\infty & \text{otherwise},
\end{cases}
\]
where $K=\reals_-^l\times\{0\}$ and $H$ is the $K$-convex random function defined by
\[
\dom H(\cdot,\omega)=\bigcap_{j=1}^m\dom f_j(\cdot,\omega)\quad\text{and}\quad H(x,\omega)=(f_i(x,\omega))_{j=1}^m.
\]
The Lagrangian integrand becomes
\begin{align*}
l(x,y,\omega) &= \inf\{f(x,u,\omega)-u\cdot y\}\\
&=\begin{cases}
+\infty & \text{if $x\notin\dom H(\cdot,\omega)$},\\
f_0(x,\omega) + y\cdot H(x,\omega) & \text{if $x\in\dom H(\cdot,\omega)$ and $y\in K^*$},\\
-\infty & \text{otherwise}
\end{cases}
\end{align*}
and the conjugate of $f$
\begin{align*}
f^*(p,y)&=\sup_{x\in\reals^n}\{x\cdot p - l(x,y)\}\\
&=
\begin{cases}
  \sup_{x\in\reals^n}\{x\cdot p-f_0(x)-y\cdot H(x)\mid x\in\dom H(\cdot,\omega)\} & \text{if $y\in K^*$},\\
  +\infty & \text{if $y\not\in K^*$}.
\end{cases}
\end{align*}
If $\dom Ef\cap(\X\times\U)\ne\emptyset$, \thref{lem:dual} says that the dual problem can be written as
\begin{equation}\label{dmp}\tag{$D_{MP}$}
\begin{aligned}
  &\maximize\ E\inf_{x\in\reals^n}\{f_0(x)+y\cdot H(x)-x\cdot p\}\ \ovr (p,y)\in\X_a^\perp\times\Y\\
  &\st\qquad\qquad y\in K^*\quad a.s.
\end{aligned}
\end{equation}
%\begin{equation}\label{dmp}\tag{$D_{MP}$}
%\maximize\ E\inf_{x\in\reals^n}\{l(x,y)-x\cdot p\}\ \ovr (p,y)\in\X_a^\perp\times\Y.
%\end{equation}
To get more explicit expressions for $f^*$ and the dual problem, additional structure is needed; see \thref{ex:lp3} below.

Recall that the {\em normal cone} of a convex set $C$ at at point $x$ is given by
\[
N_C(x):=\{v\in\reals^n\mid (x'-x)\cdot v \le 0 \quad\forall x'\in C\}.
\]
When $C$ is a convex cone, then
\begin{equation}\label{eq:compl}
v\in N_C(x)\quad \Leftrightarrow\quad x\in X,\ v\in C^*,\  x\cdot v =0,
\end{equation}
where $C^*:=\{v\in\reals^n\mid x\cdot v\le 0\ \forall x\in C\}$ is the {\em polar cone} of $C$; see the end of \cite[Section~23]{roc70a}. \thref{thm:kktr} gives the following.

\begin{theorem}\thlabel{mp3}
If $\dom Ef\cap(\X\times\U)\ne\emptyset$ and \eqref{mp} and \eqref{dmp} are feasible, then the following are equivalent
\begin{enumerate}
\item
  $x$ solves \eqref{mp}, $(p,y)$ solves \eqref{dmp} and $\inf\eqref{mp}=\sup\eqref{dmp}$,
\item
  $x$ is feasible in \eqref{mp}, $(p,y)$ is feasible in \eqref{dmp} and
\begin{equation*}\label{mpkkt3}
  \begin{gathered}
  p\in\partial_x[f_0+y\cdot H](x),\\
  H(x)\in K,\quad y\in K^*,\quad y\cdot H(x)=0
  \end{gathered}
\end{equation*}
almost surely. 
\end{enumerate}
\end{theorem}

\begin{proof}
It suffices to note that, when $(x,y)\in\dom l$, we have
\[
0\in\partial_y[-l](x,y) = -H(x) + N_{K^*}(y)
\]
if and only if $H(x)\in N_{K^*}(y)$.   This is equivalent with the given complementarity condition by \eqref{eq:compl}.
\end{proof}

\begin{assumption}\thlabel{ass:mp4}
\mbox{}
\begin{enumerate}
\item \eqref{mp} is feasible,
\item $\dom Ef\cap(\X\times\U)\ne\emptyset$,
\item $\{x\in\N\mid f^\infty_j(x)\le 0\ j=0,\dots l, f^\infty_j(x)=0\ j=l+1,\dots,m\}$ is a linear space,
\item  there exists a $p\in\X_a^\perp$ and an $\epsilon>0$ such that for all $\lambda\in(1-\epsilon,1+\epsilon)$ there exist a $y\in\Y$ and $\beta\in L^1$ such that $y\in K^*$ and 
\begin{equation*}
f_0(x,\omega)+y(\omega)\cdot H(x,\omega)\ge \lambda x\cdot p(\omega) - \beta(\omega)\quad\forall x\in\reals^n.
\end{equation*}
\end{enumerate}
\end{assumption}

\thref{thm:varphi,mp3} give the following.
\begin{theorem}\thlabel{thm:mp4}
Under \thref{ass:mp4}, $\inf\eqref{mp}=\sup\eqref{dmp}$ and $\eqref{mp}$ has a solution. In this case, a dual feasible $(p,y)$ solves \eqref{dmp} if and only if there exists a primal feasible $x$ with
\begin{equation*}\label{mpkkt3}
  \begin{gathered}
  p\in\partial_x[f_0+y\cdot H](x),\\
  H(x)\in K,\quad y\in K^*,\quad y\cdot H(x)=0
  \end{gathered}
\end{equation*}
almost surely.
\end{theorem}

 In case of linear stochastic programming, the dual can be written down explicitly in terms of the problem data.

\begin{example}[Linear stochastic programming]\thlabel{ex:lp3}
Consider the problem
\begin{alignat*}{2}
&\minimize\quad& E&[x\cdot c]\quad\ovr x\in\N\\
&\st\quad& &Ax-b\in K\quad \text{a.s.}
\end{alignat*}
and assume that there exists $(x,u)\in\X\times\U$ such that $E[x\cdot c]<\infty$ and $Ax+u-b\in K$ almost surely. The dual problem becomes
\begin{alignat*}{2}
&\minimize\quad& E&[b\cdot y]\quad\ovr\ p\in\X_a^\perp,\ y\in\Y,\\
&\st\quad& c+&A^*y = p,\ y\in K^* \quad \text{a.s.}
\end{alignat*}
and the scenariowise KKT-conditions
\begin{align*}
  \begin{gathered}
  A^*y+c=p,\\
  Ax-b \in K,\quad y\in K^*,\quad  (Ax-b)\cdot y=0,
  \end{gathered}
\end{align*}
where $A^*$ is the scenariowise transpose of $A$.

Indeed, this is a special case of \eqref{mp} with $f_0(x,\omega)=c(\omega)\cdot x$ and $f_j(x,\omega)=a_j(\omega)\cdot x-b_j(\omega)$ for $j=1,\ldots,m$. We get
\[
l(x,y,\omega) = x\cdot c(\omega) + y\cdot A(\omega)x - y\cdot b(\omega) - \delta_{K^*}(y)
\]
and
\begin{align*}
  f^*(p,y,\omega) &= \sup_{x\in\reals^n}\{x\cdot v - l(x,y,\omega)\}\\
  &=\begin{cases}
  y\cdot b(\omega) & \text{if $y\in K^*$ and $c(\omega)+A^*(\omega)y=p$},\\
  +\infty & \text{otherwise.}
  \end{cases}
\end{align*}
This gives the dual problem while the KKT conditions follow directly from \thref{mp3}. 
\end{example}

We will denote the {\em adapted projection} of an integrable process $u$ by
\[
\ap u:=(E_tu_t)_{t=0}^T.
\]

\begin{example}[Linear stochastic programming, reduced dual]
In the setting of \thref{ex:lp3} assume that $c\in\V$ and $A^*y\in\V$ for all $y\in\Y$. Then, a pair $(p,y)$ solves the dual if and only if $y$ solves the {\em reduced dual problem}
\begin{alignat*}{2}
&\minimize\quad& E&[b\cdot y]\quad\ovr\ y\in\Y,\\
&\st\quad& \ap(c+&A^*y) = 0,\ y\in K^*\quad \text{a.s.}
\end{alignat*}
and $p=c+A^*y-\ap(c+A^*y)$. If the elements of $c_t$  and the columns $A_t$ of $A$ corresponding to $x_t$ are $\F_t$-measurable, then the reduced dual can be written as
\begin{alignat*}{2}
&\minimize\quad& E&[b\cdot y]\quad\ovr\ y\in\Y,\\ 
&\st\quad& c_t+& A^*_t\cdot E_t y = 0\ t=0,\dots,T,\ y\in K^*\quad \text{a.s.}
\end{alignat*}

\end{example}
\begin{proof}
The first claim is clear and the second claim is a straightforward application of \thref{lem:ce}. 
\end{proof}

\subsection{Optimal stopping}\label{sec:os3}

Let $R$ be a real-valued adapted stochastic process and consider the {\em optimal stopping problem}
\begin{equation}\label{os}\tag{$OS$}
  \maximize\quad ER_\tau\quad\ovr \tau\in\T,
\end{equation}
where $\T$ is the set of {\em stopping times}, i.e.\ measurable functions $\tau:\Omega\to\{0,\ldots,T+1\}$ such that $\{\omega\in\Omega\mid \tau(\omega)\le t\}\in\F_t$ for each $t=0,\ldots,T$. Choosing $\tau=T+1$ is interpreted as not stopping at all. The problem
\begin{equation}\label{ros}\tag{$ROS$}
\begin{aligned}
&\maximize\quad & & E\sum_{t=0}^TR_t x_t\quad\ovr x\in\N,\\
&\st\quad & & x\ge 0,\ \sum_{t=0}^Tx_t\le 1\quad \text{a.s.}
\end{aligned}
\end{equation}
%\[
%\maximize_{x\in\N}\quad E \sum_{t=0}^T R_tx_t \quad\st\ x\ge 0,\quad \sum_{t=0}^Tx\le 1\ P\text{-a.s.}
%\]
is the convex relaxation of \eqref{os} in sense that their optimal values coincide and the extreme points of the feasible set of \eqref{ros} can be identified with $\T$; see \cite[Section~5.2]{pp22}. 

Problem \eqref{ros} fits the general duality framework with $n_t=1$, $m=1$,
\[
f(x,u,\omega) =
\begin{cases}
-\sum_{t=0}^Tx_tR_t(\omega) & \text{if $x\ge 0$ and $\sum_{t=0}^Tx_t+u\le 0$},\\
+\infty & \text{otherwise}
\end{cases}
\]
and $\bar u=-1$. We get
\begin{align*}
l(x,y,\omega) &= \inf_{u\in\reals^n}\{f(x,u,\omega)-uy\}\\
  &= \inf_{u\in\reals^n}\{-\sum_{t=0}^Tx_tR_t(\omega) - uy \mid x\ge 0,\ \sum_{t=0}^Tx_t+u\le 0\}\\
&=
\begin{cases}
  -\sum_{t=0}^Tx_tR_t(\omega) + y\sum_{t=0}^Tx_t + \delta_{\reals^n_+}(x) & \text{if $y\ge 0$},\\
  -\infty & \text{otherwise}
\end{cases}\\
&=
\begin{cases}
  \sum_{t=0}^Tx_t[y-R_t(\omega)] + \delta_{\reals^n_+}(x) & \text{if $y\ge 0$},\\
  -\infty & \text{otherwise},
\end{cases}
\end{align*}
and
\begin{align*}
  f^*(p,y,\omega) &= \sup_x\{x\cdot p - l(x,y,\omega)\}\\
  &=\sup_{x\in\reals^n_+}\sum_{t=0}^Tx_t[p_t-y+R_t(\omega)] + \delta_{\reals_+}(y)\\
  &=
  \begin{cases}
    0 & \text{if $y\ge 0$ and $p_t+R_t(\omega)\le y,\ t=0,\ldots,T$},\\
    +\infty & \text{otherwise}.
  \end{cases}
\end{align*}
Since $\dom Ef\cap(\X\times\U)\ne\emptyset$, \thref{lem:dual} says that the dual of \eqref{ros} can be written as
\begin{equation}\label{dos}\tag{$D_{OS}$}
\begin{aligned}
&\minimize\quad Ey\quad\ovr(p,y)\in\X_a^\perp\times \Y_+\\
&\st\quad p_t+R_t\le y\quad t=0,\ldots,T\ \text{a.s.}
\end{aligned}
\end{equation}
It is clear that \eqref{ros} is feasible, and \eqref{dos} is feasible when the pathwise maximum $\max_t R_t$ belongs $\Y$.  \thref{thm:kktr} thus gives the following.

\begin{theorem}\thlabel{thm:os}
If $\max_t R_t\in \Y$, then the following are equivalent,
\begin{enumerate}
\item
  $x$ solves \eqref{ros}, $(p,y)$ solves \eqref{dos} and there is no duality gap.
\item
  $x\in\X_a$ and $(p,y)\in\X_a^\perp\times\Y$ and
\begin{gather*}
  x_t\ge 0,\ p_t+R_t\le y,\ x_t(p_t+R_t-y)=0\quad t=0,\ldots,T,\\
  y\ge 0,\ \sum_{t=0}^Tx_t\le 1,\ y(\sum_{t=0}^Tx_t-1)=0
\end{gather*}
almost surely.
\end{enumerate}
In particular, a stopping time $\tau\in\T$ is optimal in \eqref{os} and $(p,y)\in\X_a^\perp\times\Y$ solves the dual if and only if $p_t+R_t\le y$ for all $t$ and $p_\tau+R_\tau = y$ almost surely.
\end{theorem}

\begin{proof}
The scenariowise KKT-condition in  \thref{thm:kktr} can be written as
\begin{align*}
  p_t+R_t-y&\in N_{\reals_+}(x_t)\quad t=0,\ldots,T,\\
  \sum_{t=0}^Tx_t-1&\in N_{\reals_+}(y),
\end{align*}
This is equivalent to the conditions given in the statement by \eqref{eq:compl}. The second claim thus follows from \thref{thm:duality,cor:kkt}. The last claim follows from the fact that a $\tau\in\T$ solves the optimal stopping problem \eqref{os} if and only if the process $x\in\X_a$ given by
\[
x_t=\begin{cases}
1 &\text{if $t=\tau$},\\
0 &\text{if $t\ne\tau$}
\end{cases}
\]
is optimal in \eqref{ros};  \cite[Section~5.2]{pp22}.
\end{proof}

\begin{example}[Reduced dual]
Assume that $E_t \Y\subseteq\Y\subseteq \V_t$ for all $t$. A pair $(p,y)\in \V\times \Y$ solves the dual problem if and only if $p_t=y-E_t y$ and the process $y_t:=E_ty$ solves the ``reduced dual''
\begin{align*}
&\minimize\quad Ey_0\quad\ovr y\in\M^\Y_+\\
&\st\quad R_t\le y_t\quad t=0,\ldots,T\ \text{a.s.},
\end{align*}
where $\M^\Y_+$ is the cone of nonnegative martingales $y$ with $y_T\in\Y$. Moreover, $x\in\X_a$ and $y\in\Y$ are primal and dual optimal, respectively, if and only if 
\begin{gather*}
  x_t\ge 0,\ R_t\le y_t,\ x_t(R_t-y_t)=0\quad t=0,\ldots,T,\\
  y_T\ge 0,\ \sum_{t=0}^Tx_t\le 1,\ y_T(\sum_{t=0}^Tx_t-1)=0.
\end{gather*}
In particular, a stopping time $\tau\in\T$ is optimal in \eqref{os} and $y\in\M^\Y_+$ solves the relaxed dual if and only if $R_t\le y_t$ for all $t$ and $R_\tau=y_\tau$ almost surely.
\end{example}

We end this section by applying the results of Section~\ref{sec:adg}. \thref{ass:adg} holds with $p=0$ and $y=\max_t R_t$. \thref{thm:varphi,thm:os} thus give the following.

\begin{theorem}
If $\max_t R_t\in \Y$, then  $\sup\eqref{os}=\sup\eqref{ros}=\inf\eqref{dos}$, and \eqref{os} and \eqref{ros} have a solution. In this case, a dual feasible $(p,y)$ solves \eqref{dos} if and only if there exists a stopping time $\tau\in\T$ with $p_\tau+R_\tau = y$ almost surely.
\end{theorem}

\subsection{Optimal control}\label{sec:ocduality}

Consider the optimal control problem
\begin{equation}\label{oc}\tag{$OC$}
\begin{aligned}
&\minimize\quad & & E\left[\sum_{t=0}^{T} L_t(X_t,U_t)\right]\quad\ovr\ (X,U)\in\N,\\
&\st\quad & & \Delta X_{t}=A_t X_{t-1} +B_t U_{t-1}+W_t\quad t=1,\dots,T
\end{aligned}
\end{equation}
where the {\em state} $X$ and the {\em control} $U$ are processes with values in $\reals^N$ and $\reals^M$, respectively, $A_t$ and $B_t$ are $\F_t$-measurable random matrices, $W_t$ is an $\F_t$-measurable random vector and the functions $L_t$ are convex normal integrands. The linear constrains in \eqref{oc} are called the {\em system equations}.

The problem fits the general duality framework with $x=(X,U)$, $\bar u=(W_t)_{t=1}^T$ and
\[
f(x,u,\omega)=\sum_{t=0}^{T}L_t(X_t,U_t,\omega) + \sum_{t=1}^{T}\delta_{\{0\}}(\Delta X_{t}- A_t(\omega)X_{t-1}-B_t(\omega)U_{t-1}-u_t).
\]
We thus assume that $\X$ and $\U$ are solid decomposable spaces of $\reals^{(T+1)(N+M)}$- and $\reals^{TM}$-valued random variables, respectively, and that $(W_1,\dots,W_T)\in\U$. By solidity,
\[
\U=\U_1\times\cdots\times\U_T,\quad\U=\Y_1\times\cdots\times\Y_T,
\]
where $\U_t$ and $\Y_t$ are solid decomposable spaces of $\reals^M$-valued random variables in separating duality under the bilinear form $(u_t,y_t)\mapsto E[u_t\cdot y_t]$. It follows that 
\[
\langle u,y\rangle=\sum_{t=1}^TE[u_t\cdot y_t].
\]
For simplicity, we assume further that, for all $t$,
\[
\begin{aligned}
  \X_t&=\S\times\C,&\quad \U_t&=\S\\
  \V_t&=\S'\times\C', &\quad \Y_t&=\S'
\end{aligned}
\]
where $\S$ and $\C$ are solid decomposable spaces in separating duality with $\S'$ and $\C'$, respectively.

%\[
%\begin{aligned}
%\X &:= (\S_0\times \C_0)\times\cdots\times(\S_T\times\C_T),  &  \U &:= \S_1\times\cdots\times\S_T,\\
%\V &:= (\S'_0\times \C'_0)\times\cdots\times(\S'_T\times\C'_T),  & \Y &:= \S'_1%\times\cdots\times\S'_T,
%\end{aligned}
%\]
%where $\S_t,\S'_t\subset L^0(\Omega,\F,P;\reals^N)$ and $\C_t,\C'_t\subset L^0(\Omega,\F,P;\reals^{M})$ are decomposable spaces paired with the bilinear forms
%\begin{align*}
%\langle X_t,X'_t\rangle &:= E[X_t\cdot X'_t]\quad\forall (X_t,X'_t)\in\S_t\times\S'_t,\\
% \langle U_t,U'_t \rangle &:= E[U_t\cdot U'_t]\quad \forall (U_t,U'_t)\in\C_t\times\C'_t.
%\end{align*}
The Lagrangian integrand becomes
\begin{align*}
l(x,y,\omega) &=\inf_{u\in\reals^m}\{f(x,u,\omega)-u\cdot y\}\\
&=\sum_{t=0}^{T} L_t(X_t,U_t,\omega) - \sum_{t=1}^{T}(\Delta X_{t}- A_t(\omega)X_{t-1}-B_t(\omega)U_{t-1})\cdot y_t\\ 
&=\sum_{t=0}^T[L_t(X_t,U_t,\omega) + X_t\cdot (\Delta y_{t+1}+A^*_{t+1}(\omega)y_{t+1}) + U_t\cdot B^*_{t+1}(\omega)y_{t+1}].
\end{align*}
The conjugate integrand can be written as
\begin{align*}
f^*(v,y,\omega) &=\sup_{x\in\reals^n}\{x\cdot v-l(x,y,\omega)\}\\
&=\sum_{t=0}^{T} L^*_t(v_t - (\Delta y_{t+1}+A^*_{t+1}(\omega)y_{t+1}, B^*_{t+1}(\omega)y_{t+1}),\omega),
\end{align*}
where $y_{T+1}:=0$, $A_{T+1}:=0$ and $B_{T+1}:=0$.

As soon as $\dom Ef\cap(\X\times\U)\ne\emptyset$, \thref{lem:dual} says that the dual problem can be written as
\begin{equation}\label{doc}\tag{$D_{OC}$}
\begin{aligned}
&\maximize & & E\left[\sum_{t=1}^{T}W_t\cdot y_t-\sum_{t=0}^{T} L^*_t(p_t - (\Delta y_{t+1}+A^*_{t+1}y_{t+1},B^*_{t+1}y_{t+1}))\right]\\% \ovr (p,y)\in\V\times\Y\\
&\ovr & & (p,y)\in\X_a^\perp\times\Y.
\end{aligned}
\end{equation}
\thref{thm:duality,cor:kkt} give the following.

\begin{theorem}\thlabel{ocdual}
If $\dom Ef\cap(\X\times\U)\ne\emptyset$ and \eqref{oc} and \eqref{doc} are feasible, then the following are equivalent
\begin{enumerate}
\item
  $(X,U)$ solves \eqref{oc}, $(p,y)$ solves \eqref{doc} and there is no duality gap,
\item
  $(X,U)$ is feasible in \eqref{oc}, $(p,y)$ is feasible in \eqref{doc} and, for all $t$, 
  \begin{align*}
    p_t-(\Delta y_{t+1}+A^*_{t+1}y_{t+1},B^*_{t+1} y_{t+1})\in\partial L_t(X_t,U_t),\\
    \Delta X_{t}=A_t X_{t-1} +B_t U_{t-1}+W_t
  \end{align*}
  almost surely.
\end{enumerate}
\end{theorem}

The optimality conditions in \thref{ocdual} can be formulated also in the form of a stochastic maximum principle.
\begin{remark}[Maximum principle]\thlabel{rem:maxoc}
The scenariowise KKT-conditions in \thref{ocdual} mean that $(X,U)$ satisfies the system equations and that
\begin{equation}\label{eq:maxoc}
-(\Delta y_{t+1},0)\in\partial_{(X_t,U_t)} H_t(X_t,U_t,y_{t+1}) - p_t,
\end{equation}
where
\[
H_t(X_t,U_t,y_{t+1}) := L_t(X_t,U_t) + y_{t+1}\cdot(A_{t+1}X_t+B_{t+1}U_t).
\]
This can be written equivalently as
\begin{align*}
U_t&\in\argmin_{U_t\in\reals^M}\{H_t(X_t,U_t,y_{t+1})-(X_t,U_t)\cdot p_t\},\\
-\Delta y_{t+1}&\in \partial_{X_t} \bar H_t(X_t,p_t,y_{t+1}),
\end{align*}
where
\begin{align*}
%H_t(X_t,U_t,p_t,y_{t+1}) &:= L_t(X_t,U_t) + y_{t+1}\cdot(A_{t+1}X_t+B_{t+1}U_t) - (X_t,U_t)\cdot p_t,\\
\bar H_t(X_t,p_t,y_{t+1}) &:= \inf_{U_t\in\reals^M}\{H_t(X_t,U_t,y_{t+1})-(X_t,U_t)\cdot p_t\}.
\end{align*}

If, for all $(X_t,U_t,y_{t+1})\in\reals^N\times\reals^M\times\reals^N$,
\begin{equation}\label{eq:3L}
\partial_{(X_t,U_t)}H_t(X_t,U_t,y_{t+1}) = \partial_{X_t} H_t(X_t,U_t,y_{t+1})\times\partial_{U_t} H_t(X_t,U_t,y_{t+1}),
\end{equation}
this can be written as
\begin{align*}
U_t&\in\argmin_{U_t\in\reals^M}\{H_t(X_t,U_t,y_{t+1})-(X_t,U_t)\cdot p_t\},\\
-\Delta y_{t+1}&\in \partial_{X_t}\{H_t(X_t,U_t,y_{t+1})-(X_t,U_t)\cdot p_t\}
%(&=\partial_XH_t(X_t,U_t,y_{t+1})-p^X_t)
\end{align*}
almost surely. Condition \eqref{eq:3L} holds, in particular, if $L_t$ is of the form
\[
L_t(X,U) = L_t^0(X,U)+L_t^1(X)+L_t^2(U),
\]
where $L_t^0$ is differentiable.
\end{remark}

\begin{proof}
The optimality conditions in \thref{ocdual} mean that
\begin{equation}\label{eq:mp}
  -(\Delta y_{t+1},0)\in\partial f_t(X_t,U_t),
\end{equation}
where $f_t(X_t,U_t) := H_t(X_t,U_t,y_{t+1}) -(X_t,U_t)\cdot p_t$. The first claim thus follows from \cite[Theorem 37.5]{roc70a}. Under \eqref{eq:3L}, condition \eqref{eq:mp} can be written as
\begin{align*}
  -\Delta y_{t+1}&\in\partial_{X_t} f_t(X_t,U_t),\\
  0&\in\partial_{U_t} f_t(X_t,U_t),
\end{align*}
which is the second condition.
\end{proof}

\begin{assumption}\thlabel{dpOCass}\ 
The spaces $\S'$ and $\C'$ are the K\"othe duals of $\S$ and $\C$, respectively, and, for all $t$,
\begin{enumerate}[label=\Alph*]
\item\label{OCassA} $E_t\S\subseteq\S$ and $E_t\C\subseteq\C$,
\item\label{OCassB} $A_{t}\S\subseteq\S$ and $B_{t}\C\subseteq\S$.
\end{enumerate}
\end{assumption}

Except for condition B, \thref{dpOCass} holds automatically e.g.\ in Lebesgue and Orlicz spaces. Part B holds e.g.\ if columns of $A_t$ and $B_t$ belong to $L^\infty$ or, alternatively, if $\C$ and $\S$ are Cartesian products of spaces of finite moments (see \cite[Section~6.1]{pp220}) and  the columns of $A_t$ and $B_t$ belong to $\S$.  By \thref{lem:EGcont3,lem:A0}, \thref{dpOCass} implies that, for all $t$,
\begin{enumerate}[label=\Alph*]
\item[A$'$] $E_t\S'\subseteq\S'$ and $E_t\C'\subseteq\C'$,
\item[B$'$] $A_{t}\S'\subseteq\S'$ and $B_{t}\S'\subseteq\C'$.
\end{enumerate}
Moreover, \thref{dpOCass} implies the following.
\begin{lemma}\thlabel{lem:EABcomm}
Under \thref{dpOCass},
\[
E_t[A_t^*y_t]=A_t^*E_ty_t\quad \text{and}\quad E_t[B_t^*y_t]=B_t^*E_ty_t
\]
for all $y_t\in\S'$.
\end{lemma}

\begin{proof}
B implies that $y_t\cdot A_tX_{t-1}$ is integrable for all $X_{t-1}\in\S$ and $y_t\in\S'$. Solidity of $\S$ implies that if we take $X_{t-1}\in\S$ and set all but one of its components to zero, the resulting vector is still in $\S$. Similarly for $\S'$. Condition B thus implies that $(A^*_t)_{i,j}y^j_t\in L^1$ for all $i,j$. The claim now follows from \thref{lem:ce}.
\end{proof}

\begin{remark}[Reduced dual]\thlabel{rem:ocrd}
Assume that each $L_t$ is $\F_t$-measurable, each $EL_t$ is proper on $\S\times\C$ and that \thref{dpOCass} holds. Then the optimum value of the dual problem \eqref{doc} equals that of the {\em reduced dual problem}
\begin{equation*}
\begin{aligned}
&\maximize\quad & & E\left[\sum_{t=1}^T W_t\cdot y_t -\sum_{t=0}^{T} [L^*_t(-E_t(\Delta y_{t+1}+A^*_{t+1}y_{t+1}, E_tB^*_{t+1}y_{t+1}))]\right]\quad\ovr\  y\in\Y_a
\end{aligned}
\end{equation*}
If $(p,y)\in\X_a^\perp\times\Y$ solves \eqref{doc}, then $\ap y$ solves the reduced dual. If $y$ solves the reduced dual and 
\[
p_t = (\Delta y_{t+1}+A^*_{t+1}y_{t+1},B^*_{t+1}y_{t+1})- E_t(\Delta y_{t+1}+A^*_{t+1}y_{t+1},B^*_{t+1}y_{t+1}),
\]
then $(p,y)$ solves \eqref{doc}.

If \eqref{doc} has a solution, then an $x$ is optimal if and only if it is feasible and there is a $y$ feasible in the reduced dual  such that 
\begin{align}\label{eq:ocrdmax}
-E_t(\Delta y_{t+1}+A^*_{t+1}y_{t+1},B^*_{t+1} y_{t+1})\in\partial L_t(X_t,U_t)
\end{align}
almost surely.
\end{remark}

The optimality conditions \eqref{eq:ocrdmax} are closely related to (5.20a)--(5.20e) in \cite{cccd15}. It should be noted however, that in \cite{cccd15}, the functions $L_t$ depend on $W_{t+1}$.

\begin{proof}
Let $y\in\Y$. By the Jensen's inequality in Theorem~\ref{thm:jensen}, assumptions imply
\begin{multline*}
  \inf_{p\in\X_a^\perp} E\sum_{t=0}^{T} L^*_t(p_t - (\Delta y_{t+1}+A^*_{t+1}y_{t+1},B^*_{t+1}y_{t+1}))\\
  = E\sum_{t=0}^{T} L^*_t(-E_t(\Delta y_{t+1}+A^*_{t+1}y_{t+1},B^*_{t+1}y_{t+1})).
  %  &= E\sum_{t=0}^{T} L^*_t(-E_t[\pi^*\Delta y_{t+1}+C^*_{t+1}y_{t+1}]).
\end{multline*}
The properties A$'$ and B$'$ stated after \thref{dpOCass} imply that the $p$ given in the statement belongs to $\X_a^\perp$ so it attains the infimum above. By \thref{lem:EABcomm},
\[
-E_t(\Delta y_{t+1}+A^*_{t+1}y_{t+1},B^*_{t+1}y_{t+1}) = -E_t(\Delta \ap y_{t+1}+A^*_{t+1}\ap y_{t+1},B^*_{t+1}\ap y_{t+1}),
\]
so $y$ can be chosen adapted without worsening the dual objective. The last claim follows from the third and \thref{ocdual}.
\end{proof}

\begin{remark}[Maximum principle in reduced form]
The scenariowise optimality condition in \thref{rem:ocrd} can be written as
\[
-(E_t\Delta y_{t+1},0)\in\partial_{(X,U)} H_t(X_t,U_t,y_{t+1}),
\]
where
\[
H_t(X_t,U_t,y_{t+1}) := L_t(X_t,U_t) + E_t[A_{t+1}^*y_{t+1}]\cdot X_t+E_t[B_{t+1}^*y_{t+1}]\cdot U_t.
\]
As in \thref{rem:maxoc}, this can be written also as
\begin{align*}
U_t&\in\argmin_{U_t\in\reals^M}H_t(X_t,U_t,y_{t+1}),\\
-E_t\Delta y_{t+1}&\in \partial_X \bar H_t(X_t,y_{t+1}),
\end{align*}
where
\begin{align*}
%H_t(X_t,U_t,p_t,y_{t+1}) &:= L_t(X_t,U_t) + y_{t+1}\cdot(A_{t+1}X_t+B_{t+1}U_t) - (X_t,U_t)\cdot p_t,\\
\bar H_t(X_t,y_{t+1}) &:= \inf_{U_t\in\reals^M}H_t(X_t,U_t,y_{t+1}).
\end{align*}
\end{remark}

We end this section by an application of \thref{thm:varphi}. In optimal control, \thref{ass:adg} holds under the following.

\begin{assumption}\thlabel{ass:oc4}
\mbox{}
\begin{enumerate}
\item \eqref{oc} is feasible,
\item $\dom Ef\cap(\X\times\U)\ne\emptyset$,
\item $\{(X,U)\in\N\mid L^\infty_t(X_t,U_t)\le 0\, \Delta X_{t}=A_t X_{t-1} +B_t U_{t-1}\   t=0,\dots T\}$ is a linear space,
\item  there exists a $p\in\X_a^\perp$ and an $\epsilon>0$ such that for all $\lambda\in(1-\epsilon,1+\epsilon)$ there exist a $y\in\Y$ such that $(\lambda p,y)$ is feasible in \eqref{doc}.
\end{enumerate}
\end{assumption}

%\thref{ass:oc4} implies the assumptions of \thref{lem:adg} so the claims follow from

\thref{thm:varphi,ocdual} give the following.

\begin{theorem}%\thlabel{ocdual}
Under \thref{ass:oc4}, $\inf\eqref{oc}=\sup\eqref{doc}$ and \eqref{oc} has a solution. In this case, a dual feasible $(p,y)$ solves \eqref{doc} if and only if there exists a primal feasible $x$ such that, for all $t$, 
  \begin{align*}
    p_t-(\Delta y_{t+1}+A^*_{t+1}y_{t+1},B^*_{t+1} y_{t+1})\in\partial L_t(X_t,U_t),\\
    \Delta X_{t}=A_t X_{t-1} +B_t U_{t-1}+W_t
  \end{align*}
  almost surely.
\end{theorem}

\subsection{Problems of Lagrange}\label{sec:lagrange3}

Consider the problem
\begin{equation}\label{lagrange}\tag{$L$}
\minimize\quad E\sum_{t=0}^T K_t(x_t,\Delta x_t)\quad\ovr x\in\N,
\end{equation}
where $x$ is a process of fixed dimension $d$, $K_t$ are convex normal integrands and $x_{-1}:=0$. Problem~\eqref{lagrange} can be thought of as a discrete-time version of a problem studied in calculus of variations. Other problem formulations have $K_t(x_{t-1},\Delta x_t)$ instead of $K_t(x_t,\Delta x_t)$ in the objective, or an additional term of the form $Ek(x_0,x_T)$, all of which fit the general format of stochastic optimization.

This fits the general duality framework with $\bar u=0$ and
\[
f(x,u,\omega) = \sum_{t=0}^T K_t(x_{t},\Delta x_t+u_t,\omega).
\]
We thus assume that both $\X$ and $\U$ are solid decomposable spaces of $\reals^{(T+1)d}$-valued random variables. For simplicity, we assume that 
\[
\X_t=\S,\quad \V_t=\S',\quad \U=\X,\quad\Y=\V,
\]
where $\S$ and $\S'$ are solid decomposable spaces in separating duality.

The Lagrangian integrand becomes
\begin{align*}
  l(x,y,\omega) &= \sum_{t=0}^T\left[\Delta x_t\cdot y_t+H_t(x_{t},y_t,\omega))\right]\\
&=\sum_{t=0}^T\left[-x_{t}\cdot\Delta y_{t+1}+H_t(x_{t},y_t,\omega)\right],  
\end{align*}
where $y_{T+1}:=0$ and
\[
H_t(x_{t},y_t,\omega) := \inf_{u_t\in\reals^d}\{K_t(x_{t},u_t,\omega)-u_t\cdot y_t\}
\]
is the associated {\em Hamiltonian}. The conjugate integrand can be written as
\begin{align*}
  f^*(v,y,\omega) &= \sup\{x\cdot v-l(x,y,\omega)\}\\
  &=\sum_{t=0}^T K_t^*(v_{t}+\Delta y_{t+1},y_t,\omega).
\end{align*}
If \eqref{lagrange} is feasible, \thref{lem:dual} says that the dual problem can be written as
\begin{equation}\label{dl}\tag{$D_L$}
\maximize\quad\quad E [-\sum_{t=0}^T K_t^*(p_{t}+\Delta y_{t+1},y_t)]\quad\ovr y\in\Y,p\in \X_a^\perp
\end{equation}
where $y_{T+1} := 0$. \thref{thm:duality,cor:kkt} now give the following.

\begin{theorem}\thlabel{thm:dl}
If $\dom Ef\cap(\X\times\U)\ne\emptyset$ and \eqref{lagrange} and \eqref{dl} are feasible, then the following are equivalent
\begin{enumerate}
  \item
    $x$ solves \eqref{lagrange}, $(p,y)$ solves \eqref{dl} and there is no duality gap,
  \item
    $x$ is feasible in \eqref{lagrange}, $(p,y)$ is feasible in \eqref{dl} and
    \begin{align*}
      \begin{split}
        p_{t} + \Delta y_{t+1}&\in \partial_{x} H_t(x_{t},y_t),\\
        \Delta x_t &\in \partial_{y}[-H_t](x_{t},y_t),
      \end{split}
    \end{align*}
    almost surely.
\end{enumerate}
\end{theorem}

Note that, by \cite[Theorem 37.5]{roc70a}, the scenariowise KKT-conditions can be written equivalently as the discrete-time stochastic Euler--Lagrange equations
\begin{align}\label{eq:eloc}
(p_{t} + \Delta y_{t+1},y_t)\in \partial K_t(x_{t},\Delta x_t)
\end{align}
or
\begin{align*}
(x_{t},\Delta x_t)\in \partial K^*_t(p_{t} + \Delta y_{t+1},y_t).
\end{align*}

\begin{assumption}\thlabel{ass:L3}
The space $\S'$ is the K\"othe dual of $\S$ and $E_t\S\subseteq\S$ for all $t$.
\end{assumption}

By \thref{lem:EGcont3}, \thref{ass:L3} implies that $E_t\S'\subset\S'$ for all $t$.

\begin{remark}[Reduced dual]\thlabel{rdl}
Consider \thref{thm:dl} and assume that \thref{ass:L3} holds and that, for all~$t$, $K_t$ is $\F_t$-measurable and $EK_t$ is proper on $\S\times\S$. Then the optimum value of the dual problem \eqref{dl} equals that of the reduced dual problem
\begin{equation*}
\begin{aligned}
&\maximize\quad\quad  & & E [-\sum_{t=0}^T K_t^*(E_t\Delta y_{t+1},y_t)]\quad\ovr y\in\Y_a.
\end{aligned}
\end{equation*}
If $(p,y)$ solves the dual problem, then $\ap y$ solves the reduced dual problem. If $y$ solves the reduced dual problem, then $(p_t,y_t)=(E_ty_{t+1}-y_{t+1},y_t)$ solves the dual problem.

If \eqref{dl} has a solution, then an $x$ is optimal if and only if it is feasible and there is a $y$ feasible in the reduced dual  such that 
 \begin{align*}
      \begin{split}
        E_t\Delta y_{t+1}&\in \partial_{x} H_t(x_{t},y_t),\\
        \Delta x_t &\in \partial_{y}[-H_t](x_{t},y_t),
      \end{split}
    \end{align*}
almost surely
\end{remark}

\begin{proof}
Given $(p,y)\in\X_a^\perp\times\Y$, the assumptions imply, by Jensen's inequality in Theorem~\ref{thm:jensen},
\begin{align*}
 E [\sum_{t=0}^T K_t^*(p_{t}+\Delta y_{t+1},y_t)]\ge &E [\sum_{t=0}^T K_t^*(E_{t}\Delta y_{t+1},E_{t}y_t)].
\end{align*}
Given $y\in\Y_a$, the above holds as an equality when $p_t=E_{t}\Delta y_{t+1}-\Delta y_{t+1}$. Since $p\in\X_a^\perp$ under \thref{ass:L3}, this proves the first claim, which together with \thref{thm:dl} implies the next two. The last claim follows from the third and \thref{thm:dl}.
\end{proof}

\begin{remark}
Any $y\in\Y_a$ has the {\em Doob decomposition}
\[
y_t=m_t+a_t,
\]
where $m$ is a martingale and $a_t$ is $\F_{t-1}$-measurable. Indeed, let $\Delta a_t := E_{t-1}\Delta y_{t}$, $a_0:=0$, $\Delta m_t := \Delta y_t-\Delta a_t$ and $m_0:=y_0$. The reduced dual problem in \eqref{rdl} can be written as
\begin{equation*}
\begin{aligned}
&\maximize\quad & & E [-\sum_{t=0}^T K_t^*(\Delta a_{t+1},m_t+a_t)]\quad\ovr\ (m,a)\in (\Y_a\cap\M)\times\Y_p\\
&\st & & a_t\in\F_{t-1}, t=1,\dots,T, a_0=0,
\end{aligned}
\end{equation*}
where $\M$ is the set of martingales and $\Y_p$ is the of predictable processes. If the reduced dual problem has a solution, then $x$ is optimal if and only if there exists an $y$ feasible in the reduced dual such that
\begin{align*}%\label{eq:lagrangeredoc}
\begin{split}
\Delta a_{t+1}&\in \partial_{x} H_t(x_{t},m_t+a_t),\\
\Delta x_t &\in \partial_{y}[-H_t](x_{t},m_t+a_t),
\end{split}
\end{align*}
where $m+a$ is the Doob decomposition of $y$.
\end{remark}

\begin{example}[Optimal stopping]
The relaxed optimal stopping problem
\begin{equation*}\tag{$ROS$}
\maximize_{x\in\N_+}\quad E \sum_{t=0}^TR_t\Delta x_t\quad\st\quad \Delta x\ge 0,\ x_T\le 1\ a.s.
\end{equation*}
from Section~\ref{sec:os3} can be written as a problem of Lagrange with $d=1$ and
\[
K_t(x_t,u_t) = -R_tu_t + \delta_{\reals_-}(x_t-1) + \delta_{\reals_+}(u_t).
\]
We get
\begin{align*}
  K_t^*(v_t,y_t) &= \sup_{x_t,u_t\in\reals}\{x_t\cdot v_t+u_t\cdot y_t-K_t(x_t,u_t)\}\\
  &= \sup_{x_t,u_t\in\reals}\{x_t\cdot v_t+u_t\cdot y_t+R_tu_t\mid x_t\le 1,\ u_t\ge 0\}\\
  &=\begin{cases}
 v_t  & \text{if $v_t\ge 0$ and $R_t+y_t\le 0$},\\
+\infty & \text{otherwise}
  \end{cases}
\end{align*}
so the reduced dual becomes
\[
\begin{aligned}
  &\maximize & E y_0\ \ovr &y\in\Y_a\\
  &\st & E_t[\Delta y_{t+1}]&\ge 0,\\
  & & R_t+y_t&\le 0,
\end{aligned}
\]
or with the change of variables $S:=-y$,
\[
\begin{aligned}
  &\minimize\quad &  ES_0\ \ovr &S\in\Y_a\\
  &\st\quad & E_t[\Delta S_{t+1}]&\le 0,\\
  &\quad\quad & R_t&\le S_t.
\end{aligned}
\]
Thus, feasible dual solutions are supermartingales that dominate the reward process $R$.
%The Snell envelope of the positive part of $R$ (see Section~\ref{ssec:os2}) solves the dual problem. Indeed, if the dual had a solution $S$ better than the Snell envelope, then $S$ would be a supermartingale that dominates the positive part of $R$ (recall that $y_{T+1}:=0$) and it would be strictly smaller than the Snell envelope at $t=0$.

The Hamiltonian can be written as
\begin{align*}
  H_t(x,y) &= \sup_{u\in\reals}\{uy-K_t(x_t,u_t)\}\\
  &=
  \begin{cases}
    +\infty & \text{if $x_t>0$},\\
    0 & \text{if $x_t\le 1$ and $R_t+y_t\le 0$},\\
    -\infty & \text{otherwise}
  \end{cases}
\end{align*}
so the optimality conditions become
\begin{align*}
  E_t\Delta y_{t+1}&\in N_{\reals_-}(x_t-1)\\
  \Delta x_t&\in N_{\reals_-}(R_t+y_t).
\end{align*}
This implies that is $\Delta x_t$ nonzero only when $R_t=-y_t$ and $E_t\Delta y_{t+1}$ is nonzero only when $x_t=1$.  
\end{example}

We end this section by an application of \thref{thm:varphi}.

\begin{assumption}\thlabel{ass:l4}
\mbox{}
\begin{enumerate}
\item \eqref{lagrange} is feasible,
\item $\dom Ef\cap(\X\times\U)\ne\emptyset$,
\item $\{x\in\N\mid \sum_{t=0}^T K^\infty_t(x_t,\Delta x_t)\le 0\}$ is a linear space,
\item there exists a $p\in\X_a^\perp$ and an $\epsilon>0$ such that for all $\lambda\in(1-\epsilon,1+\epsilon)$ there exist a $y\in\Y$ such that $(\lambda p,y)$ is feasible in \eqref{dl}.
\end{enumerate}
\end{assumption}

\thref{ass:l4} implies \thref{ass:adg} so \thref{thm:varphi,thm:dl} give the following.

\begin{theorem}%\thlabel{ocdual}
 Under \thref{ass:l4}, $\inf\eqref{lagrange}=\sup\eqref{dl}$ and \eqref{oc} has a solution. In this case, a dual feasible $(p,y)$ solves \eqref{dl} if and only if there exists a primal feasible $x$ such that, for all $t$,
  \begin{align*}
      \begin{split}
        p_{t} + \Delta y_{t+1}&\in \partial_{x} H_t(x_{t},y_t),\\
        \Delta x_t &\in \partial_{y}[-H_t](x_{t},y_t),
      \end{split}
    \end{align*}
    almost surely.
\end{theorem}

\subsection{Financial mathematics}\label{sec:fm3}

Let $s=(s_t)_{t=0}^T$ be an adapted $\reals^J$-valued stochastic process describing the unit prices of a finite set $J$ of perfectly liquid tradeable assets. Assume also that there is a finite set $\bar J$ of derivative assets that can be bought or sold at time $t=0$ and that provide random payments $\bar c^j\in L^0$, $j\in\bar J$ at time $t=T$. We denote $\bar c=(\bar c^j)_{j\in\bar J}$. The cost of buying a derivative portfolio $\bar x\in\reals^{\bar J}$ at the best available market prices is denoted by $S_0(\bar x)$. Such a function is convex; see e.g.\ \cite{pen11,pen11b}. We will assume, for simplicity, that $S_0$ is finite on $\reals^{\bar J}$. For example, if the buying and selling prices of the derivative assets are given by vectors $s^b\in\reals^{\bar J}$ and $s^a\in\reals^{\bar J}$ of bid- and ask-prices, respectively, and if we assume that one can buy and sell infinite quantities at these prices, then
\[
S_0(\bar x)=\sup_{s\in[s^b,s^a]}\bar x\cdot s.
\]
%If the bid and ask prices come with finite quantities given by vectors $q^b\in\reals^{\bar J}$ and $q^a\in\reals^{\bar J}$, respectively, then
%\[
%S_0(\bar x)=\sup_{s\in[s^b,s^a]}\bar x\cdot s + \delta_{[-q^b,q^a]}(\bar x). 
%\]

Consider the problem of finding a dynamic trading strategy $x=(x_t)_{t=0}^T$ in the liquid assets $J$ and a static portfolio $\bar x$ in the derivatives $\bar J$ so that their combined revenue provides the ``best hedge'' against the financial liability of delivering a random amount $c\in L^0$ of cash at time $T$. If assume that cash (or another numeraire asset) is a perfectly liquid asset that can be lent and borrowed at zero interest rate, the problem can be written as
\begin{equation*}\label{ssh}\tag{$SSH$}
\begin{aligned}
&\minimize\quad & & EV\left(c - \sum_{t=0}^{T-1}x_t\cdot\Delta s_{t+1} - \bar c\cdot \bar x + S_0(\bar x)\right)\ \ovr\ x\in\N,\bar x\in\reals^{\bar J},\\
  &\st\quad & & x_t \in D_t\quad t=0,\ldots,T-1\ a.s.,
\end{aligned}
\end{equation*}
where $V:\reals\to\ereals$ is a nondecreasing nonconstant convex ``loss function'' and $D_t$ is a random $\F_t$-measurable set describing possible portfolio constraints. We will assume $D_T=\{0\}$, which means that all positions have to be closed at the terminal date. Note that nondecreasing convex loss functions $V$ are in one-to-one correspondence with nondecreasing concave utility functions $U$ via $V(c)=-U(-c)$; see e.g.\ \cite[Section~8.2]{fs16}. 

The special case where there are no statically held derivative assets, i.e.\ $\bar J=\emptyset$, has been extensively studied in the literature of financial mathematics; see e.g.~\cite{rs5} its references. In the literature on ``model-independent'' mathematical finance, problems of finding both the dynamically updated portfolio $x$ and the static part $\bar x$ are often referred to as ``semi-static hedging''; see e.g.\ \cite{bhp13}. One should note, however, that problem \eqref{ssh} is based on the assumption that one can buy and sell arbitrary quantities of the assets $J$ at prices given by $s$. It also assumes that one can lend and borrow arbitrary amounts of cash at zero interest rate. Under these assumptions, the random variable $c$ can be thought of as the difference of the claim to be hedged and the initial wealth and the sum in the objective can be interpreted as the proceeds from trading the assets $J$ over the period $[0,T]$. More realistic models for dynamic trading have been analyzed in \cite{pen11,pen11b,pp18d}.

As soon as $c\in\U$, problem~\eqref{ssh} fits the general duality framework with the time index running from $-1$ to $T-1$, $\F_{-1}=\{\Omega,\emptyset\}$, $x_{-1}=\bar x$, $\bar u=c$ and
\[
f(x,u,\omega)=V\left(u - \sum_{t=0}^{T-1}x_t\cdot\Delta s_{t+1}(\omega) - \bar c(\omega)\cdot \bar x + S_0(\bar x),\omega\right)+ \sum_{t=0}^{T-1}\delta_{D_t(\omega)}(x_t,\omega).
\]
%We will assume that, for all $t$,
%\[
%\X_t=\S\quad\text{and}\quad\V_t=\S',
%\]
%where $\S$ and $\S'$ are solid decomposable spaces in separating duality.
The Lagrangian integrand becomes
\begin{align*}
  l(x,y,\omega) &= \inf_{u\in\reals}\{f(x,u,\omega)-uy\}\\
  &=y\left[S_0(\bar x) - \bar c(\omega)\cdot \bar x - \sum_{t=0}^{T-1}x_t\cdot\Delta s_{t+1}(\omega)\right] - V^*(y,\omega) + \sum_{t=0}^{T}\delta_{D_t(\omega)}(x_t)
\end{align*}
and the conjugate of $f$,
\begin{align*}
  f^*(v,y,\omega) &= \sup_{x\in\reals^n}\{x\cdot v - l(x,y,\omega)\}\\
  &= V^*(y,\omega) + \sum_{t=0}^{T-1}\sigma_{D_t(\omega)}(v_t+y\Delta s_{t+1}(\omega)) + \sup_{\bar x\in\reals^{\bar J}}\{\bar x\cdot(v_{-1} + y\bar c(\omega)) - yS_0(\bar x)\}\\
  &= V^*(y,\omega) + \sum_{t=0}^{T-1}\sigma_{D_t(\omega)}(v_t+y\Delta s_{t+1}(\omega)) + (yS_0)^*(v_{-1}+y\bar c(\omega)).%\sigma_{\epi S_0}(v_{-1}+y\bar c(\omega),-y).
\end{align*}

It is natural to assume that $S_0(0)=0$ and $0\in D_t$ almost surely for all $t$. If $EV$ is proper on $\U$, \thref{lem:dual} then says that the dual problem can be written as
\begin{equation}\label{dssh}\tag{$D_{SSH}$}
  \begin{aligned}
    &\maximize_{p\in \X_a^\perp, y\in\Y}\quad & & E\left[cy - V^*(y) - \sum_{t=0}^{T-1}\sigma_{D_t}(p_t+y\Delta s_{t+1}) - (yS_0)^*(p_{-1}+y\bar c)\right].
  \end{aligned}
\end{equation}
\thref{thm:duality,cor:kkt} give the following.

\begin{theorem}\thlabel{thm:ssh}
If \eqref{ssh} and \eqref{dssh} are feasible, then the following are equivalent
\begin{enumerate}
\item
  $(\bar x,x)$ solves \eqref{ssh}, $(p,y)$ solves \eqref{dssh} and there is no duality gap.
\item
  $(\bar x,x)$ is feasible in \eqref{ssh}, $(p,y)$ is feasible in \eqref{dssh} and
\begin{align*}
y &\in\partial V(u-\sum_{t=0}^{T-1} x_t\Delta s_{t+1}-\bar c\cdot\bar x+S_0(\bar x)),\\
p_t+y\Delta s_{t+1} &\in N_{D_t}(x_t)\quad t=0,\dots,T,\\
p_{-1}+y\bar c&\in\partial(yS_0)(\bar x)
\end{align*}
almost surely.
\end{enumerate}
\end{theorem}

Under the following assumption, the dual problem \eqref{dssh} can be written in a reduced form where the shadow price of information $p$ has been optimized for a each given $y$.

\begin{assumption}\thlabel{ass:fm3}
$\X=L^\infty$, $\V=L^1$, $\Y$ is the K\"othe dual of $\U$ and, for all $t$,
  \begin{enumerate}
  \item[A]
    $E_t\U\subseteq\U$,
  \item[B]
    $\Delta s_{t+1}\in\U$.
  \end{enumerate}
\end{assumption}

If part B holds, \thref{ass:fm3} holds, e.g., in Lebesgue and Orlicz spaces; see the examples in Section~\ref{sec:dsrv}. By \thref{lem:EGcont3,lem:A0}, \thref{ass:fm3} implies that, for all~$t$,
\begin{enumerate}[label=\Alph*]
\item[A$'$] $E_t\Y\subseteq\Y$,
\item[B$'$] $y\Delta s_{t+1}\in L^1$ for all $y\in\Y$.
\end{enumerate}

\begin{remark}[Reduced dual]\thlabel{rdssh}
Consider \thref{thm:ssh} and assume that \thref{ass:fm3} holds. Then the optimum value of the dual problem \eqref{ssh} equals that of the reduced dual problem
\begin{equation}\label{rdssh}\tag{$rD_{SSH}$}
  \begin{aligned}
    &\maximize_{y\in\Y}\quad & & E\left[cy - V^*(y) - \sum_{t=0}^{T-1}\sigma_{D_t}(E_t[y\Delta s_{t+1}]) - (E[y]S_0)^*(E[y\bar c])\right].
  \end{aligned}
\end{equation}
Moreover, $y$ solves \eqref{rdssh} if and only if $(p,y)$ solves \eqref{dssh}, where
\[
p_{-1}:=\frac{E[y\bar c]}{E[y]}y-y\bar c\quad\text{and}\quad p_t=E_t[y\Delta s_{t+1}]-y\Delta s_{t+1}\quad t=0,\ldots,T-1.
\]

If \eqref{dl} has a solution, there is no duality gap and $\inf\eqref{ssh}>E[\inf V]$, then an $x$ is optimal if and only if it is feasible and there is a $y$ feasible in the reduced dual  such that 
\begin{align*}
y &\in\partial V(u-\sum_{t=0}^{T-1} x_t\Delta s_{t+1}-\bar c\cdot\bar x+S_0(\bar x)),\\
E_t[y\Delta s_{t+1}] &\in N_{D_t}(x_t)\quad t=0,\dots,T,\\
\frac{E[y\bar c]}{E[y]}&\in\partial S_0(\bar x)
\end{align*}
almost surely.
\end{remark}
\begin{proof}
If $(p,y)$ is feasible in \eqref{dssh}, then by Jensen's inequality, the objective of \eqref{rdssh} minorizes that of \eqref{dssh}. On the other hand, if $y$ is feasible in \eqref{rdssh} with $E[y]>0$ and we $p$ as in the statement, then $(p,y)$ is feasible in \eqref{dssh} and the objective values coincide. Indeed, by sublinearity, this choice gives
\begin{align*}
(yS_0)^*(p_{-1}+y\bar c) &= (yS_0)^*(\frac{E[y\bar c]}{E[y]}y)\\
&= (\frac{y}{E[y]}E[y]S_0)^*(\frac{y}{E[y]}E[y\bar c])\\
&= \frac{y}{E[y]}(E[y]S_0)^*(E[y\bar c])
\end{align*}
If $E[y]=0$ then $y=0$ almost surely and $(p,y)=(0,0)$ is feasible in \eqref{dssh} and gives the same objective value. By \thref{thm:ssh}, dual optimal $y$ is necessarily nonzero unless $V$ achieves its minimum almost surely at
\[
u-\sum_{t=0}^{T-1} x_t\Delta s_{t+1}-\bar c\cdot\bar x+S_0(\bar x).
\]
The last claim follows from \thref{thm:ssh}. 
\end{proof}

The scenariowise optimality conditions above can also be written as
\begin{align*}
\lambda\frac{dQ}{dP} &\in\partial V(u-\sum_{t=0}^{T-1} x_t\Delta s_{t+1}-\bar c\cdot\bar x+S_0(\bar x)),\\
E^Q_t[\Delta s_{t+1}] &\in N_{D_t}(x_t)\quad t=0,\dots,T,\\
E^Q[\bar c]&\in\partial S_0(\bar x),
\end{align*}
where $\lambda=E[y]$ and $Q$ is the probability measure defined by $dQ/dP=y/E[y]$.
%Note that 
%\begin{align*}
%\sigma_{\epi S_0}(v,-\beta) &= \sup\{x\cdot v-\alpha\beta\mid S_0(x)\le\alpha\}\\
%&= \begin{cases}
%  +\infty & \beta<0,\\
%  \sup\{x\cdot v-\beta S_0(x)\} & \beta\ge 0
%  \end{cases}\\
%&=
%\begin{cases}
%  +\infty & \beta<0,\\
%  \beta S_0^*(v/\beta) & \beta>0,\\
%\sigma_{\dom S_0}(v) & \beta=0.
%\end{cases}
%\end{align*}
If $S_0$ is sublinear, the reduced dual can be written as
\begin{equation*}%\label{rdssh}\tag{$rD_{SSH}$}
  \begin{aligned}
    &\maximize_{y\in\Y}\quad\quad & &E\left[c y- V^*(y) - \sum_{t=0}^{T-1}\sigma_{D_t}(E_t[y\Delta s_{t+1}]) \right]\\
    &\st & &\qquad E[y\bar c]\in E[y]\dom S_0^*,
  \end{aligned}
\end{equation*}
where the set $E[y]\dom S_0^*$ is interpreted as the recession cone of $\dom S_0^*$ if $E[y]=0$. If $E[y]\ne 0$, the constraint means that
\[
E^Q\bar c \in\dom S_0^*.
\]
%where $Q$ is the probability measure defined by $dQ/dP=y/Ey$.
The constraints in the dual thus require that the measure $Q$ be ``calibrated'' to the observed market prices of the claims $\bar c$. For example, if infinite quantities are available to buy and sell at prices $s^a\in\reals^{\bar J}$ and $s^b\in\reals^{\bar J}$, respectively, then $\dom S_0^*=[s^b,s^a]$.

It turns out that, in the absence of portfolio constraints, the linearity condition in \thref{ass:adg} becomes the classical {\em no-arbitrage} condition
\begin{equation}\tag{NA}\label{na}
x\in\N,\ \sum_{t=0}^{T-1} x_t\cdot\Delta s_{t+1}\ge 0\ a.s.\implies\sum_{t=0}^{T-1} x_t\cdot\Delta s_{t+1}= 0\ a.s.;
\end{equation}
see \cite[Section~5.5]{pp22}. The lower bound in \thref{ass:adg} holds, in particular, if there exists a martingale measure $Q\ll P$ such that
\[
dQ/dP\in\Y\cap\dom EV^*,
\]
$V$ is deterministic and either
\begin{equation*}
\limsup_{u\to-\infty} \frac{uV'(u)}{V(u)}<1\quad\text{or}\quad \liminf_{u\to+\infty} \frac{uV'(u)}{V(u)}> 1;
\end{equation*}
see \cite[Remark~53]{pp22}. More generally, \thref{ass:adg} is implied by the following.

\begin{assumption}\thlabel{ass:fm4}
\mbox{}
\begin{enumerate}
\item
The set
\[
\L:=\{(\bar x,x)\in\reals^{\bar J}\times\N \mid  S_0^\infty(\bar x) - \bar x\cdot \bar c - \sum_{t=0}^{T-1} x_t\Delta s_{t+1} \le 0\}
\]
is a linear space.
%$\{x\in\N\mid \ S_0^\infty (\bar x)-\sum_{t=0}^{T-1} x_t\Delta s_{t+1}\le 0,\ x_t\in D^\infty_t\ t=0,\dots T\}$ is a linear space,
\item
There exists $y$ feasible in the reduced dual \eqref{rdssh} and $\epsilon$ such that $\lambda y \in\dom EV^*$ for all $\lambda\in(1-\epsilon,1+\epsilon)$.
  %there exists a $p\in\X_a^\perp$ and an $\epsilon>0$ such that for all $\lambda\in(1-\epsilon,1+\epsilon)$ there exist a $y\in\Y$ such that $(\lambda p,y)$ is feasible in \eqref{dalm}.
\end{enumerate}
\end{assumption}

\begin{theorem}\thlabel{thm:fm4}
Under \thref{ass:fm4}, $\bar\varphi$ is closed and the infimum in its definition is attained for every $(z,u)\in\X\times\U$. In particular, $\inf\eqref{ssh}=\sup\eqref{dssh}$ and \eqref{ssh} has a solution. In this case, a dual feasible $(p,y)$ solves \eqref{dssh} if and only if there is a primal feasible $x$ such that
\begin{align*}
y &\in\partial V(u-\sum_{t=0}^{T-1} x_t\Delta s_{t+1}-\bar c\cdot\bar x+S_0(\bar x)),\\
p_t+y\Delta s_{t+1} &\in N_{D_t}(x_t)\quad t=0,\dots,T,\\
p_{-1}+y\bar c&\in\partial(yS_0)(\bar x)
\end{align*}
almost surely.
\end{theorem}

\begin{proof}
By \cite[Theorem 9.3]{roc70a} and \cite[Theorem 7.3]{pen99},
\[
f^\infty(x,u,\omega)=V^\infty\left(u - \sum_{t=0}^{T-1}x_t\cdot\Delta s_{t+1}(\omega) - \bar c(\omega)\cdot \bar x + S_0^\infty(\bar x),\omega\right)+ \sum_{t=0}^{T-1}\delta_{D_t^\infty(\omega)}(x_t,\omega).
\]
Since $V$ is nonconstant and nondecreasing,
\[
\{c\in\reals\mid V^\infty(c)\le 0\}=\reals_-
\]
so the linearity condition in \thref{ass:fm4} implies the one in \thref{ass:adg}. For $y$ from \thref{ass:fm4}, $p$ defined by
\[
p_{-1}:=\frac{E[y\bar c]}{E[y]}y-y\bar c\quad\text{and}\quad p_t=E_t[y\Delta s_{t+1}]-y\Delta s_{t+1}\quad t=0,\ldots,T-1,
\]
satisfies \thref{ass:adg}. Thus the claims follow from \thref{thm:varphi,thm:ssh}.
\end{proof}

Given a convex function $g$ on $\reals^n$, the set
\[
\lin g=\{x\in\reals^n \mid g^\infty(x)=-g^\infty(-x)\}
\]
is called the {\em lineality space} of $g$.

\begin{example}[Robust no-arbitrage condition]
Assume that there are no portfolio constraints and that there exists a cost function $\tilde S_0$ such that
\begin{align*}
  \tilde S_0(x,\omega)&\le S_0^\infty(x,\omega)\quad\forall x\in\reals^{\bar J},\\
  \tilde S_0(x,\omega)&< S_0^\infty(x,\omega)\quad\forall x\notin\lin S_0(\cdot,\omega)
\end{align*}
and the market model described by $\tilde S_0$ and $s$ satisfies the no-arbitrage condition
\begin{equation}\label{na}
\C \cap L^0_+=\{0\},
\end{equation}
where
\[
\C:=\{c\in\U\mid \exists (\bar x,x)\in\reals^{\bar J}\times\N :\, \sum_{t=0}^{T-1}x_t\cdot\Delta s_{t+1}+\bar x\cdot c-\tilde S_0(\bar x)\ge c\quad a.s.\}.
\]
Then the linearity condition in \thref{ass:fm4} holds. A violation of \eqref{na} would mean that there is a trading strategy $(\bar x,x)$ that superhedges a nonzero nonnegative claim $c$.
\end{example}

\begin{proof}
If the linearity condition fails, there is a $(\bar x,x)\in\L$ such that 
\[
S_0^\infty(-\bar x) + \bar x\cdot \bar c + \sum_{t=0}^{T-1} x_t\Delta s_{t+1}>0
\]
on a set $A\in\F$ with $P(A)>0$. It suffices to show that $(\bar x,x)$ is an arbitrage strategy for $\tilde S_0$. Since $(\bar x,x)\in\L$, and $\tilde S_0\le S_0^\infty$, we have
\[
\tilde S_0(\bar x) - \bar x\cdot \bar c - \sum_{t=0}^{T-1} x_t\Delta s_{t+1} \le 0.
\]
If $\bar x\notin\lin S_0$, then $\tilde S_0(x)< S_0^\infty(x)$ and the inequality is strict so $(\bar x,x)$ is an arbitrage strategy. If $\bar x\in\lin S_0$, then $\tilde S_0(\bar x)\le S_0^\infty(\bar x)=-S_0^\infty(-\bar x)$ so
\[
  \tilde S_0(\bar x) - \bar x\cdot \bar c - \sum_{t=0}^{T-1} x_t\Delta s_{t+1} \le -S_0^\infty(-\bar x) - \bar x\cdot \bar c - \sum_{t=0}^{T-1} x_t\Delta s_{t+1}<0
\]
on $A$ so $(\bar x,x)$ is an arbitrage strategy in this case too.
\end{proof}

\section{Appendix}

Given extended real-valued random variables $\xi_1$ and $\xi_2$, their pointwise sum $\xi_2+\xi_2$ is well-defined by the usual algebraic operations of the extended real-line $\ereals$ except when one of them takes the value $+\infty$ and  the other one $-\infty$. In this exceptional case, we define the sum as $+\infty$.

The proofs of the following two lemmas can be found in \cite{pp22}.

\begin{lemma}\thlabel{lem:Esum}
Given extended real-valued random variables $\xi_1$ and $\xi_2$, we have
\[
E[\xi_1+\xi_2]=E[\xi_1]+E[\xi_2]
\]
under any of the following:
\begin{enumerate}
\item\label{int1}
  $\xi_1^+,\xi_2^+\in L^1$ or $\xi_1^-,\xi_2^-\in L^1$.
  \item\label{int2}
  $\xi_1\in L^1$ or $\xi_2\in L^1$,
\item\label{int4}
  $\xi_1$ or $\xi_2$ is $\{0,+\infty\}$-valued.
\end{enumerate}
\end{lemma}

\begin{lemma}\thlabel{lem:ce}
Let $\xi_1$ and $\xi_2$ be extended real-valued random variables.
\begin{enumerate}
\item
  If $\xi_1$ and $\xi_2$ are quasi-integrable and satisfy any of the conditions in Lemma~\ref{lem:Esum}, then $\xi_1+\xi_2$ is quasi-integrable and
  \[
  E^\G[\xi_1+\xi_2] = E^\G [\xi_1] + E^\G[\xi_2].
  \]
\item
  If $\xi_2$ and $(\xi_1\xi_2)$ are quasi-integrable, and $\xi_1$ is $\G$-measurable, then
  \[
  E^\G[\xi_1\xi_2] = \xi_1E^\G[\xi_2].
  \]
\end{enumerate}
\end{lemma}

\bibliographystyle{plain}
\bibliography{sp}

\end{document}